\documentclass[12pt,letterpaper,reqno]{amsart}

\usepackage{times}
\usepackage[T1]{fontenc}
\usepackage{mathrsfs}
\usepackage{latexsym}
\usepackage[dvips]{graphics}
\usepackage{epsfig}
\usepackage{amsmath,amsfonts,amsthm,amssymb,amscd}
\input amssym.def
\input amssym.tex

\newcommand{\supp}{{\rm supp}}

\newcommand{\hg}{\widehat{g}}

\newcommand{\hphi}{\widehat{\phi}}  

\newcommand\be{\begin{equation}}
\newcommand\ee{\end{equation}}
\newcommand\bea{\begin{eqnarray}}
\newcommand\eea{\end{eqnarray}}
\newcommand\bi{\begin{itemize}}
\newcommand\ei{\end{itemize}}
\newcommand\ben{\begin{enumerate}}
\newcommand\een{\end{enumerate}}
\newcommand\bc{\begin{center}}
\newcommand\ec{\end{center}}
\newcommand\ba{\begin{array}}
\newcommand\ea{\end{array}}

\def\notdiv{\ \mathbin{\mkern-8mu|\!\!\!\smallsetminus}}

\newcommand{\R}{\ensuremath{\mathbb{R}}}
\newcommand{\C}{\ensuremath{\mathbb{C}}}

\newcommand{\fof}{\frac{1}{4}}  
\newcommand{\foh}{\frac{1}{2}}  

\newcommand{\jsgeneral}[2]{{\underline{#1}\choose #2}}

\newtheorem{thm}{Theorem}[section]

\newtheorem{lem}[thm]{Lemma}

\theoremstyle{definition}
\newtheorem{rek}[thm]{Remark}

\newcommand{\twocase}[5]{#1 \begin{cases} #2 & \text{{\rm #3}}\\ #4
&\text{{\rm #5}} \end{cases}   }

\newcommand{\gep}{\epsilon}

\newcommand{\gam}[2]{\Gamma\left(\frac{#1}{#2}\right)}
\newcommand{\intii}{\int_{-\infty}^\infty}
\newcommand{\ci}{\frac1{2\pi i}}

\begin{document}

\title[A Symplectic Test of the $L$-Functions Ratios Conjecture]{A
Symplectic Test of the $L$-Functions Ratios Conjecture}

\author{Steven J. Miller}
\email{sjmiller@math.brown.edu} \address{Department of Mathematics,
Brown University, Providence, RI 02912}


\date{\today}

\thanks{The author would like to thank Eduardo Due\~nez, Chris Hughes, Duc Khiem Huynh,
Jon Keating, Nina Snaith and Sergei Treil for many enlightening
conversations, Jeffrey Stopple for finding a typo in the proof of
Lemma \ref{lem:ntsumkeven1}, and the University of Bristol for its
hospitality (where much of this work was done). This work was partly
supported by NSF grant DMS0600848.}

\begin{abstract} Recently Conrey, Farmer and Zirnbauer
\cite{CFZ1,CFZ2} conjectured formulas for the averages over a family
of ratios of products of shifted $L$-functions. Their $L$-functions
Ratios Conjecture predicts both the main and lower order terms for
many problems, ranging from $n$-level correlations and densities to
mollifiers and moments to vanishing at the central point. There are
now many results showing agreement between the main terms of number
theory and random matrix theory; however, there are very few
families where the lower order terms are known. These terms often
depend on subtle arithmetic properties of the family, and provide a
way to break the universality of behavior. The $L$-functions Ratios
Conjecture provides a powerful and tractable way to predict these
terms. We test a specific case here, that of the $1$-level density
for the symplectic family of quadratic Dirichlet characters arising
from even fundamental discriminants $d \le X$. For test functions
supported in $(-1/3, 1/3)$ we calculate all the lower order terms up
to size $O(X^{-1/2+\gep})$ and observe perfect agreement with the
conjecture (for test functions supported in $(-1, 1)$ we show
agreement up to errors of size $O(X^{-\gep})$ for any $\gep$). Thus
for this family and suitably restricted test functions, we
completely verify the Ratios Conjecture's prediction for the
$1$-level density.
\end{abstract}


\maketitle

\setcounter{equation}{0}

\setcounter{equation}{0}

\section{Introduction}

Montgomery's \cite{Mon} analysis of the pair correlation of zeros of
$\zeta(s)$ revealed a striking similarity to the behavior of
eigenvalues of ensembles of random matrices. Since then, this
connection has been a tremendous predictive aid to researchers in
number theory in modeling the behavior of zeros and values of
$L$-functions, ranging from spacings between adjacent zeros
\cite{Mon,Hej,Od1,Od2,RS} to moments of $L$-functions
\cite{CF,CFKRS}. Katz and Sarnak \cite{KaSa1,KaSa2} conjectured
that, in the limit as the conductors tend to infinity, the behavior
of the normalized zeros near the central point agree with the
$N\to\infty$ scaling limit of the normalized eigenvalues near $1$ of
a subgroup of $U(N)$. One way to test this correspondence is through
the $n$-level density of a family $\mathcal{F}$ of $L$-functions
$L(s,f)$; we concentrate on this statistic in this paper. The
$n$-level density is
\begin{eqnarray}
D_{n,\mathcal{F}}(\phi)\ :=\ \frac{1}{|\mathcal{F}|} \sum_{f\in
\mathcal{F}} \sum_{\substack{\ell_1,\dots, \ell_n \\ \ell_i \neq \pm
\ell_k}} \phi_1\left(\gamma_{f,\ell_1}\frac{\log
Q_f}{2\pi}\right)\cdots \phi_n\left(\gamma_{f,\ell_n}\frac{\log
Q_f}{2\pi}\right),
\end{eqnarray} where the $\phi_i$ are even Schwartz test functions
whose Fourier transforms have compact support, $\foh +
i\gamma_{f,\ell}$ runs through the non-trivial zeros of $L(s,f)$,
and $Q_f$ is the analytic conductor of $f$. As the $\phi_i$ are even
Schwartz functions, most of the contribution to
$D_{n,\mathcal{F}}(\phi)$ arises from the zeros  near the central
point; thus this statistic is well-suited to investigating the
low-lying zeros.

There are now many examples where the main term in number theory
agrees with the Katz-Sarnak conjectures (at least for suitably
restricted test functions), such as all Dirichlet characters,
quadratic Dirichlet characters, $L(s,\psi)$ with $\psi$ a character
of the ideal class group of the imaginary quadratic field
$\mathbb{Q}(\sqrt{-D})$, families of elliptic curves, weight $k$
level $N$ cuspidal newforms, symmetric powers of ${\rm GL}(2)$
$L$-functions, and certain families of ${\rm GL}(4)$ and ${\rm
GL}(6)$ $L$-functions (see
\cite{DM1,FI,Gu,HR,HM,ILS,KaSa2,Mil1,OS2,RR,Ro,Rub1,Yo2}).

For families of $L$-functions over function fields, the
corresponding classical compact group can be identified through the
monodromy. While the situation is less clear for $L$-functions over
number fields, there has been some recent progress. Due\~nez and
Miller \cite{DM2} show that for sufficiently nice families and
sufficiently small support, the main term in the $1$-level density
is determined by the first and second moments of the Satake
parameters, and a symmetry constant (which identifies the
corresponding classical compact group) may be associated to any nice
family such that the symmetry constant of the Rankin-Selberg
convolution of two families is the product of the symmetry
constants.

There are two avenues for further research. The first is to increase
the support of the test functions, which often leads to questions of
arithmetic interest (see for example Hypothesis S in \cite{ILS}).
Another is to identify lower order terms in the $1$-level density,
which is the subject of this paper. The main term in the $1$-level
density is independent of the arithmetic of the family, which
surfaces in the lower order terms. This is very similar to the
Central Limit Theorem. For nice densities the distribution of the
normalized sample mean converges to the standard normal. The main
term is controlled by the first two moments (the mean and the
variance of the density) and the higher moments surface in the rate
of convergence. This is similar to our situation, where the
universal main terms arise from the first and second moments of the
Satake parameters.

There are now several families where lower order terms have been
isolated in the $1$-level density \cite{FI,Mil2,Mil3,Yo1}; see also
\cite{BoKe}, where the Hardy-Littlewood conjectures are related to
lower order terms in the pair correlation of zeros of $\zeta(s)$
(see for example \cite{Be,BeKe,CS2,Ke} for more on lower terms of
correlations of Riemann zeros). Recently Conrey, Farmer and
Zirnbauer \cite{CFZ1,CFZ2} formulated conjectures for the averages
over families of $L$-functions of ratios of products of shifted
$L$-functions, such as \bea & & \sum_{d \le X}
\frac{L\left(\foh+\alpha,\chi_d\right)}{L\left(\foh+\gamma,\chi_d\right)}
 \ = \  \sum_{d \le X} \Bigg[
\frac{\zeta(1+2\alpha)}{\zeta(1+\alpha+\gamma)} A_D(\alpha;\gamma)
\nonumber\\ & & \ \ \ \ \ \ \ \ \ \ \ + \
\left(\frac{d}{\pi}\right)^{-\alpha}
\frac{\Gamma\left(\frac14-\frac{\alpha}2\right)}{\Gamma\left(\frac14+\frac{\alpha}2\right)}
\frac{\zeta(1-2\alpha)}{\zeta(1-\alpha+\gamma)}A_D(-\alpha;\gamma)\Bigg]
\ + \ O(X^{1/2 + \gep}) \eea (here $d$ ranges over even fundamental
discriminants, $-1/4 < \Re(\alpha) < 1/4$, $1/\log X \ll \Re(\gamma)
< 1/4$, and $A_D$ (we only give the definition for $\alpha =
\gamma$, as that is the only instance that occurs in our
applications) is defined in \eqref{eq:defnADAD'}). Their
$L$-functions Ratios Conjecture arises from using the approximate
functional equation, integrating term by term, and retaining only
the diagonal pieces (which they then `complete'); they also assume
uniformity in the parameters so that the resulting expressions may
be differentiated (this is an essential ingredient for $1$-level
density calculations). It is worth noting the incredible detail of
the conjecture, predicting all terms down to $O(X^{1/2+\gep})$.

There are many difficult computations whose answers can easily be
predicted through applications of the $L$-functions Ratios
Conjecture, ranging from $n$-level correlations and densities to
mollifiers and moments to vanishing at the central point (see
\cite{CS1}). While these are not proofs, it is extremely useful for
researchers to have a sense of what the answer should be. One common
difficulty in the subject is that often the number theory and random
matrix theory answers appear different at first, and much effort
must be spent on combinatorics to prove agreement (see for example
\cite{Gao,HM,Rub1,RS}); the analysis is significantly easier if one
knows what the final answer should be. Further, the Ratios
Conjecture often suggest a more enlightening way to group terms (see
for instance Remark \ref{rek:inflzeroszetazerosdirich}).

Our goal in this paper is to test the predictions of the Ratios
Conjecture for a specific family, that of quadratic Dirichlet
characters. We let $d$ be a fundamental discriminant. This means
(see \S5 of \cite{Da}) that either $d$ is a square-free number
congruent to 1 modulo 4, or $d/4$ is square-free and congruent to 2
or 3 modulo 4. If $\chi_d$ is the quadratic character associated to
the fundamental discriminant $d$, then if $\chi_d(-1) = 1$ (resp.,
$-1$) we say $d$ is even (resp., odd). If $d$ is a fundamental
discriminant then it is even (resp., odd) if $d > 0$ (resp., $d<0$).
We concentrate on even fundamental discriminants below, though with
very few changes our arguments hold for odd discriminants (for
example, if $d$ is odd there is an extra $1/2$ in certain Gamma
factors in the explicit formula).\\

\emph{For notational convenience we adopt the following conventions
throughout the paper:} \bi \item Let $X^\ast$ denote the number of
even fundamental discriminants at most $X$; thus $X^\ast = 3X/\pi^2
+ O(X^{1/2})$, and $X/\pi^2 + O(X^{1/2})$ of these have $4|d$ (see
Lemma \ref{lem:numdwithpdivided} for a proof).
\item In any sum over $d$, $d$ will range over even fundamental
discriminants unless otherwise specified. \ei

The goal of these notes is to calculate the lower order terms (on
the number theory side) as much as possible, as unconditionally as
possible, and then compare our answer to the prediction from the
$L$-functions Ratios Conjecture, given in the theorem below.

\begin{thm}[One-level density from the Ratios Conjecture \cite{CS1}]\label{thm:oneldRatios}
Let $g$ be an even Schwartz test function such that $\hg$ has finite
support. Let $X^\ast$ denote the number of even fundamental
discriminants at most $X$, and let $d$ denote a typical even
fundamental discriminant. Assuming the Ratios Conjecture for
$\sum_{d \le X} L(\foh+\alpha,\chi_d)/L(\foh+\gamma,\chi_d)$, we
have \bea\label{eq:ratios1lddirchar} & & \frac1{X^\ast}\sum_{d \le X
} \sum_{\gamma_d} g\left(\gamma_d \frac{\log
X}{2\pi}\right)\nonumber\\ & & = \ \frac1{X^\ast\log X} \intii
g(\tau) \sum_{d \le X} \Bigg[ \log\frac{d}{\pi} + \foh
\frac{\Gamma'}{\Gamma}\left(\frac14+\frac{i\pi \tau}{\log X}\right)+
\foh \frac{\Gamma'}{\Gamma}\left(\frac14-\frac{i\pi \tau}{\log
X}\right)\Bigg]d\tau\nonumber\\ & & \ \ + \ \frac2{X^\ast\log
X}\sum_{d \le X} \intii g(\tau) \Bigg[\frac{\zeta'}{\zeta}\left(1+
\frac{4\pi i\tau}{\log X}\right) + A_D'\left(\frac{2\pi i \tau}{\log
X};\frac{2\pi i \tau}{\log X}\right) \nonumber\\ & & \ \ -\ e^{-2\pi
i \tau \log(d/\pi)/\log X} \frac{\Gamma\left(\frac14-\frac{\pi i
\tau}{\log X}\right)}{\Gamma\left(\frac14+\frac{\pi i \tau}{\log
X}\right)}\ \zeta\left(1 - \frac{4\pi i \tau}{\log
X}\right)A_D\left(-\frac{2\pi i \tau}{\log X};\frac{2\pi i
\tau}{\log X}\right)\Bigg] d\tau \nonumber\\ & & \ \ +\
O(X^{-\foh+\gep}), \eea with \bea\label{eq:defnADAD'} A_D(-r,r) & \
= \ & \prod_p \left(1 - \frac1{(p+1)p^{1-2r}} -
\frac1{p+1}\right) \cdot \left(1 - \frac1p\right)^{-1} \nonumber\\
A_D'(r;r) &=& \sum_p \frac{\log p}{(p+1)(p^{1+2r}-1)}. \eea The
above is \be \frac1{X^\ast}\sum_{d \le X} \sum_{\gamma_d}
g\left(\gamma_d \frac{\log X}{2\pi}\right)  \ = \
\int_{-\infty}^\infty g(x) \left(1 - \frac{\sin(2\pi x)}{2\pi
x}\right) dx + O\left(\frac{1}{\log X}\right), \ee which is the
$1$-level density for the scaling limit of ${\rm USp}(2N)$. If
$\supp(\hg) \subset (-1,1)$, then the integral of $g(x)$ against
$-\sin(2\pi x)/2\pi x$ is $-g(0)/2$.

If we assume the Riemann Hypothesis, for $\supp(\hg) \subset
(-\sigma, \sigma) \subset (-1, 1)$ we have \bea & &
\frac{-2}{X^\ast\log X} \sum_{d\le X} \intii g(\tau)\ e^{-2\pi i
\tau \frac{\log(d/\pi)}{\log X}} \frac{\Gamma\left(\frac14-\frac{\pi
i \tau}{\log X}\right)}{\Gamma\left(\frac14+\frac{\pi i \tau}{\log
X}\right)}\ \zeta\left(1 - \frac{4\pi i \tau}{\log
X}\right)A_D\left(-\frac{2\pi i \tau}{\log X};\frac{2\pi i
\tau}{\log X}\right) d\tau\nonumber\\ & & \ \ \ \ \ \ \ \ = \
-\frac{g(0)}2 + O(X^{-\frac34(1-\sigma)+\gep}); \eea the error term
may be absorbed into the $O(X^{-1/2+\gep})$ error in
\eqref{eq:ratios1lddirchar} if $\sigma < 1/3$.
\end{thm}

The conclusions of the above theorem are phenomenal, and demonstrate
the power of the Ratios Conjecture. Not only does its main term
agree with the Katz-Sarnak conjectures for arbitrary support, but it
calculates the lower order terms up to size $O(X^{-1/2+\gep})$.
While Theorem \ref{thm:oneldRatios} is conditional on the Ratios
Conjecture, the following theorem is not, and provides highly
non-trivial support for the Ratios Conjecture.

\begin{thm}[One-level density for quadratic Dirichlet
characters]\label{thm:1ldquadevenfdfromNT} Let the notation be as in
Theorem \ref{thm:oneldRatios}, with $\supp(\hg) \subset (-\sigma,
\sigma)$. \ben \item Up to terms of size
$O(X^{-(1-\sigma)/2+\gep})$, the $1$-level density for the family of
quadratic Dirichlet characters with even fundamental discriminants
at most $X$ agrees with \eqref{eq:ratios1lddirchar} (the prediction
from the Ratios
Conjecture). 
\item If we instead consider the family $\{8d:\ 0 < d \le X, \ d$
an odd, positive square-free fundamental discriminant$\}$, then the
$1$-level density agrees with the prediction from the Ratios
Conjecture
up to terms of size $O(X^{-1/2}+X^{-(1 - \frac32\sigma)+\gep} +
X^{-\frac34(1-\sigma)+\gep})$. In particular, if $\sigma < 1/3$ then
the number theory calculation agrees with the Ratios Conjecture up
to errors at most $O(X^{-1/2+\gep})$.\een
\end{thm}

\begin{rek}
The above theorem indicates that, at least for the family of
quadratic Dirichlet characters and suitably restricted test
functions, the Ratios Conjecture \emph{is} predicting all lower
order terms up to size $O(X^{-1/2 +\gep})$. This is phenomenal
agreement between theory and conjecture. Previous investigations of
lower order terms in $1$-level densities went as far as $O(\log^N
X)$ for some $N$; here we are getting square-root agreement, and
strong evidence in favor of the Ratios Conjecture. 
\end{rek}

\begin{rek}[Influence of zeros of $\zeta(s)$ on lower order
terms]\label{rek:inflzeroszetazerosdirich} From the expansion in
\eqref{eq:ratios1lddirchar} we see that one of the lower order terms
(arising from the integral of $g(\tau)$ against $\zeta'(1+4\pi i
\tau/\log X)/\zeta(1+4\pi i \tau/\log X)$) in the $1$-level density
for the family of quadratic Dirichlet characters is controlled by
the non-trivial zeros of $\zeta(s)$. This phenomenon has been noted
by other researchers (Bogomolny, Conrey, Keating, Rubinstein,
Snaith); see \cite{CS1,BoKe,HKS,Rub2} for more details, especially
\cite{Rub2} for a plot of the influence of zeros of $\zeta(s)$ on
zeros of $L$-functions of quadratic Dirichlet characters.
\end{rek}

The proof of Theorem \ref{thm:1ldquadevenfdfromNT} starts with the
Explicit Formula, which relates sums over zeros to sums over primes
(for completeness a proof is given in Appendix
\ref{sec:explicitformula}). For convenience to researchers
interested in odd fundamental discriminants, we state it in more
generality than we need.

\begin{thm}[Explicit Formula for a family of Quadratic Dirichlet
Characters]\label{thm:oneldNT} Let $g$ be an even Schwartz test
function such that $\hg$ has finite support. For $d$ a fundamental
discriminant let $a(\chi_d) = 0$ if $d$ is even ($\chi_d(-1) = 1$)
and $1$ otherwise. Consider a family $\mathcal{F}(X)$ of fundamental
discriminants at most $X$ in absolute value. We have \bea & &
\frac1{|\mathcal{F}(X)|}\sum_{d \in \mathcal{F}(X)} \sum_{\gamma_d}
g\left(\gamma_d \frac{\log X}{2\pi}\right)\nonumber\\ & & =\
\frac1{|\mathcal{F}(X)|\log X} \intii g(\tau) \sum_{d \in
\mathcal{F}(X)} \Bigg[ \log\frac{|d|}{\pi} + \foh
\frac{\Gamma'}{\Gamma}\left(\frac14+\frac{a(\chi_d)}2+\frac{i\pi
\tau}{\log X}\right)\nonumber\\ & & \ \ \ \ + \foh
\frac{\Gamma'}{\Gamma}\left(\frac14+\frac{a(\chi_d)}2-\frac{i\pi
\tau}{\log X}\right)\Bigg]d\tau - \frac{2}{|\mathcal{F}(X)|} \sum_{d
\in \mathcal{F}(X)} \sum_{k=1}^\infty \sum_p \frac{\chi_d(p)^k\log
p}{p^{k/2}\log X}\ \hg\left(\frac{\log p^k}{\log X}\right).\nonumber\\
\eea
\end{thm}

As our family has only even fundamental discriminants, all
$a(\chi_d)=0$. The terms arising from the conductors (the
$\log(|d|/\pi)$ and the $\Gamma'/\Gamma$ terms) agree with the
Ratios Conjecture. We are reduced to analyzing the sums of
$\chi_d(p)^k$ and showing they agree with the remaining terms in the
Ratios Conjecture. As our characters are quadratic, this reduces to
understanding sums of $\chi_d(p)$ and $\chi_d(p)^2$. We first
analyze the terms from the Ratios Conjecture in
\S\ref{sec:ratiosconjterms} and then we analyze the character sums
in \S\ref{sec:numbthterms}. We proceed in this order as one of the
main uses of the Ratios Conjecture is in predicting simple forms of
the answer; in particular, it suggests non-obvious simplifications
of the number theory sums.

\setcounter{equation}{0}

\section{Analysis of the terms from the Ratios
Conjecture.}\label{sec:ratiosconjterms}

We analyze the terms in the $1$-level density from the Ratios
Conjecture (Theorem \ref{thm:oneldRatios}). The first piece
(involving $\log(d/\pi)$ and $\Gamma'/\Gamma$ factors) is already
matched with the terms in the Explicit Formula arising from the
conductors and $\Gamma$-factors in the functional equation. In
\S\ref{sec:numbthterms} we match the next two terms (the integral of
$g(\tau)$ against $\zeta'/\zeta$ and $A_D'$) to the contributions
from the sum over $\chi_d(p)^{k}$ for $k$ even; we do this for test
functions with arbitrary support. The number theory is almost equal
to this; the difference is the presence of a factor $-g(0)/2$ from
the even $k$ terms, which we match to the remaining piece from the
Ratios Conjecture.

This remaining piece is the hardest to analyze. We denote it by
\bea\label{eq:ratiospiecenotinNT} R(g;X) & \ = \ &
-\frac2{X^\ast\log X}\sum_{d \le X} \intii g(\tau)
 e^{-2\pi i \tau \frac{\log(d/\pi)}{\log X}}
\frac{\Gamma\left(\frac14-\frac{\pi i \tau}{\log
X}\right)}{\Gamma\left(\frac14+\frac{\pi i \tau}{\log
X}\right)}\nonumber\\ & & \ \ \ \ \ \ \cdot \ \zeta\left(1 -
\frac{4\pi i \tau}{\log X}\right)A_D\left(-\frac{2\pi i \tau}{\log
X};\frac{2\pi i \tau}{\log X}\right)d\tau, \eea with (see
\eqref{eq:defnADAD'}) \bea A_D(-r,r) & \ = \ & \prod_p \left(1 -
\frac1{(p+1)p^{1-2r}} - \frac1{p+1}\right) \cdot \left(1 -
\frac1p\right)^{-1}. \eea

There is a contribution to $R(g;X)$ from the pole of $\zeta(s)$. The
other terms are at most $O(1/\log X)$; however, if the support of
$\hg$ is sufficiently small then these terms contribute
significantly less.

\begin{lem}\label{lem:contrRgX} Assume the Riemann Hypothesis.
If $\supp(\hg) \subset (-\sigma, \sigma)$ then \be R(g;X)\ =\
-\frac{g(0)}2 + O(X^{-\frac34(1-\sigma)+\gep}).\ee In particular, if
$\sigma < 1/3$ then $R(g;X) = -\foh g(0) +O(X^{-\foh + \gep})$.
\end{lem}

\begin{rek} If we do not assume the Riemann Hypothesis we may prove
a similar result. The error term is replaced with
$O(X^{-(1-\frac{\theta}2)(1-\sigma) + \gep})$, where $\theta$ is the
supremum of the real parts of zeros of $\zeta(s)$. As $\theta \le
1$, we may always bound the error by $O(X^{-(1-\sigma)/2+\gep})$.
Interestingly, this is the error we get in analyzing the number
theory terms $\chi(p)^k$ with $k$ odd by applying Jutila's bound
(see \S\ref{sec:analyzingSoddJutila}); we obtain a better bound of
$O(X^{-(1-\frac32\sigma)})$ by using Poisson summation to convert
long character sums to shorter ones (see
\S\ref{sec:analyzingSoddPoissonSum}).
\end{rek}

\begin{rek} The proof of Lemma \ref{lem:contrRgX} follows from
shifting contours and keeping track of poles of ratios of Gamma and
zeta functions. We can prove a related result with significantly
less work. Specifically, if for $\supp(\hg) \subset (-1,1)$ we are
willing to accept error terms of size $O(\log^{-N} X)$ for any $N$
then we may proceed as follows: (1) modify Lemma
\ref{lem:partialsumdexpdpiX} to replace the $d$-sum with $X^\ast
e^{-2\pi i \left(1-\frac{\log \pi}{\log X}\right) \tau} \left(1 -
\frac{2\pi i \tau}{\log X}\right)^{-1} + O(X^{1/2})$; (2) use the
decay properties of $g$ to restrict the $\tau$ sum to $|\tau| \le
\log X$ and then Taylor expand everything but $g$, which gives a
small error term and \bea & & \int_{|\tau| \le \log X} g(\tau)
\sum_{n=-1}^N \frac{a_n}{\log^n X} (2\pi i \tau)^n e^{-2\pi i
\left(1-\frac{\log \pi}{\log X}\right) \tau} d\tau \nonumber\\ & &\
\ \ \ \ \ \ \ \ \ \ \ = \ \sum_{n=-1}^N \frac{a_n}{\log^n X}
\int_{|\tau| \le \log X}(2\pi i \tau)^n g(\tau) e^{-2\pi i
\left(1-\frac{\log \pi}{\log X}\right) \tau}d\tau; \eea (3) use the
decay properties of $g$ to extend the $\tau$-integral to all of $\R$
(it is essential here that $N$ is fixed and finite!) and note that
for $n \ge 0$ the above is the Fourier transform of $g^{(n)}$ (the
$n$\textsuperscript{th} derivative of $g$) at $1 - \frac{\pi}{\log
X}$, and this is zero if $\supp(\hg) \subset (-1,1)$.
\end{rek}

We prove Lemma \ref{lem:contrRgX} in \S\ref{sec:analysisRgX}; this
completes our analysis of the terms from the Ratios Conjecture. We
analyze the lower order term of size $1/\log X$ (present only if
$\supp(\hg) \not\subset (-1,1)$) in Lemma
\ref{lem:secondarytermratioslastfactor} of
\S\ref{sec:secondarytermRgX}. We explicitly calculate this
contribution because in many applications all that is required are
the main and first lower order terms. One example of this is that
zeros at height $T$ are modeled not by the $N\to\infty$ scaling
limits of a classical compact group but by matrices of size $N \sim
\log (T/2\pi)$ \cite{KeSn1,KeSn2}. In fact, even better agreement is
obtained by changing $N$ slightly due to the first lower order term
(see \cite{BBLM,DHKMS}).

\subsection{Analysis of $R(g;X)$}\label{sec:analysisRgX}

Before proving Lemma \ref{lem:contrRgX} we collect several useful
facts.

\begin{lem}\label{lem:usefulfacts} In all statements below $r = 2\pi i \tau/\log
X$ and $\supp(\hg) \subset (-\sigma, \sigma) \subset (-1,1)$. \ben
\item $A_D(-r,r) = \zeta(2)/\zeta(2-2r)$.
\item If $|r| \ge \gep$ then $|\zeta(-3-2r)/\zeta(-2-2r)|
\ll_\gep (1+|r|)$.  \item For $w \ge 0$, $g\left(\tau-iw \frac{\log
X}{2\pi}\right) \ll X^{\sigma w} \left(\tau^2 + (w\frac{\log
X}{2\pi})^2\right)^{-B}$ for any $B \ge 0$. \item For $0 < a < b$ we
have $|\Gamma(a\pm i y)/\Gamma(b\pm i y)| = O_{a,b}(1)$. \een
\end{lem}

\begin{proof} (1): From simple
algebra, as we may rewrite each factor as \be \frac{p}{p+1}
\left(1-\frac1{p^{2-2r}}\right) \frac{p}{p-1} \ = \
\left(1-\frac1{p^2}\right)^{-1} \left(1-\frac1{p^{2-2r}}\right). \ee

(2): By the functional equations of the Gamma and zeta functions
$\Gamma(s/2)\pi^{-s/2}\zeta(s)$ $=$ $\Gamma((1-s)/2)\pi^{-(1-s)/2}
\zeta(1-s)$ and $\Gamma(1+x)= x\Gamma(x)$ gives \be
\frac{\zeta(-3-2r)}{\zeta(-2-2r)} \ = \
\frac{\Gamma(1-(-1-r))\pi^{-2-r}\Gamma(-1-r)\pi^{1+r}
\zeta(4+2r)}{\Gamma(-\frac32-r)\pi^{\frac32+r}\Gamma(1-(-\frac32-r))(\frac32+r)^{-1}
\pi^{-\frac32+r} \zeta(3+2r)}. \ee Using \be \Gamma(x)\Gamma(1-x)\
=\ \pi / \sin \pi x \ = \ 2\pi i /(e^{i\pi x} - e^{-i\pi x}), \ee we
see the ratio of the Gamma factors have the same growth as $|r| \to
\infty$ (if $r=0$ then there is a pole from the zero of $\zeta(s)$
at $s=-2$), and the two zeta functions are bounded away from $0$ and
infinity.

(3): As $g(\tau) = \int \hg(\xi) e^{2\pi i \xi \tau} d\xi$, we have
\bea g(\tau-iy) & \ = \ & \intii \hg(\xi) e^{2\pi i (\tau - iy)\xi}
d\xi \nonumber\\ & = & \intii \hg^{(2n)}(\xi) (2\pi i (\tau
-iy))^{-n} e^{2\pi i (\tau - iy)\xi} d\xi \nonumber\\ & \ll &
e^{2\pi y \sigma} (\tau -iy))^{-2n}; \eea the claim follows by
taking $y = (w \log X)/2\pi$.

(4): As $|\Gamma(x-iy)|=|\Gamma(x+iy)|$, we may assume all signs are
positive. The claim follows from the definition of the Beta
function: \be \frac{\Gamma(a+iy)\Gamma(b-a)}{\Gamma(b+iy)} \ = \
\int_0^1 t^{a+iy-1} (1-t)^{b-a-1} \ = \ O_{a,b}(1); \ee see
\cite{ET} for additional estimates of the size of ratios of Gamma
functions.
\end{proof}

\begin{proof}[Proof of Lemma \ref{lem:contrRgX}] By Lemma
\ref{lem:usefulfacts} we may replace $A_D(-2\pi i \tau/\log X,2\pi i
\tau/\log X)$ with $\zeta(2)/\zeta(2-4\pi i \tau/\log X$). We
replace $\tau$ with $\tau - iw\frac{\log X}{2\pi}$ with $w=0$ (we
will shift the contour in a moment). Thus \bea R(g;X) & \ = \ &
-\frac2{X^\ast\log X}\sum_{d \le X} \intii g\left(\tau -
iw\frac{\log X}{2\pi}\right)
 e^{-2\pi i \left(\tau - iw\frac{\log X}{2\pi}\right) \frac{\log(d/\pi)}{\log X}}
\nonumber\\ & & \ \ \ \ \ \ \cdot \
\frac{\Gamma\left(\frac14-\frac{w}2-\frac{\pi i \tau}{\log
X}\right)}{\Gamma\left(\frac14+\frac{w}2+\frac{\pi i \tau}{\log
X}\right)}\ \frac{\zeta(2) \zeta\left(1 - w- \frac{4\pi i \tau}{\log
X}\right)}{\zeta\left(2 - 2w- \frac{4\pi i \tau}{\log X}\right)}\
d\tau. \eea We now shift the contour to $w = 2$. There are two
different residue contributions as we shift (remember we are
assuming the Riemann Hypothesis, so that if $\zeta(\rho) = 0$ then
either $\rho = \foh + i\gamma$ for some $\gamma \in \R$ or $\rho$ is
a negative even integer), arising from

\bi \item the pole of $\zeta\left(1 - w- \frac{4\pi i \tau}{\log
X}\right)$ at $w=\tau=0$; \item the zeros of $\zeta\left(2 - 2w-
\frac{4\pi i \tau}{\log X}\right)$ when $w=3/4$ and $\tau = \gamma
\frac{\log X}{4\pi}$ \ei (while potentially there is a residue from
the pole of $\Gamma\left(\frac14-\frac{w}2-\frac{\pi i \tau}{\log
X}\right)$ when $w=1/2$ and $\tau=0$, this is canceled by the pole
of $\zeta\left(2 - 2w- \frac{4\pi i \tau}{\log X}\right)$ in the
denominator).\\

We claim the contribution from the pole of $\zeta\left(1 - w-
\frac{4\pi i \tau}{\log X}\right)$ at $w=\tau=0$ is $-g(0)/2$. As
$w=\tau=0$, the $d$-sum is just $X^\ast$. As the pole of $\zeta(s)$
is $1/(s-1)$, since $s = 1 - \frac{4\pi i \tau}{\log X}$ the
$1/\tau$ term from the zeta function has coefficient $-\frac{\log
X}{4\pi i}$. We lose the factor of $1/2\pi i$ when we apply the
residue theorem, there is a minus sign outside the integral and
another from the direction we integrate (we replace the integral
from $-\gep$ to $\gep$ with a semi-circle oriented clockwise; this
gives us a minus sign as well as a factor of $1/2$ since we only
have half the contour), and everything else evaluated at $\tau = 0$
is $g(0)$.\\

We now analyze the contribution from the zeros of $\zeta(s)$ as we
shift $w$ to $2$. Thus $w = 3/2$ and we sum over $\tau = \gamma
\frac{\log X}{4\pi}$ with $\zeta(\foh+i\gamma) = 0$. We use Lemma
\ref{lem:partialsumdexpdpiX} (with $z = \tau - i w \frac{\log
X}{2\pi}$) to replace the $d$-sum with \be X^\ast e^{-2\pi i
\left(1-\frac{\log \pi}{\log X}\right) \tau} \left(\frac14 -
\frac{2\pi i \tau}{\log X}\right)^{-1} X^{-\frac34} X^{\frac{2\log
\pi}{\log X}} + O(\log X).\ee The contribution from the $O(\log X)$
term is dwarfed by the main term (which is of size $X^{1/4+\gep}$).
From (3) of Lemma \ref{lem:usefulfacts} we have \be g\left(\gamma
\frac{\log X}{4\pi}-i\frac34\frac{\log X}{2\pi}\right) \ \ll \
X^{3\sigma/4} (\tau^2+1)^{-B}\ee for any $B > 0$. From (4) of Lemma
\ref{lem:usefulfacts}, we see that the ratio of the Gamma factors is
bounded by a power of $|\tau|$ (the reason it is a power is that we
may need to shift a few times so that the conditions are met; none
of these factors will every vanish as we are not evaluating at
integral arguments). Finally, the zeta function in the numerator is
bounded by $|\tau|^2$. Thus the contribution from the critical zeros
of $\zeta(s)$ is bounded by \be \sum_{\gamma \atop
\zeta(\foh+i\gamma) = 0} \ \frac{1}{X^\ast \log X} \cdot X^{1/4}
\cdot \frac{X^{3\sigma/4}}{(\gamma^2+1)^B} \cdot (|\gamma \log X| +
1)^n. \ee For sufficiently large $B$ the sum over $\gamma$ will
converge. This term is of size $O(X^{-\frac34(1-\sigma) + \gep})$.
This error is $O(X^{-\gep})$ whenever $\sigma < 1$, and if $\sigma <
1/3$ then the error is at most $O(X^{-1/2+\gep})$.\\

The proof is completed by showing that the integral over $w=2$ is
negligible. We use Lemma \ref{lem:partialsumdexpdpiX} (with $z =
\tau - i 2 \frac{\log X}{2\pi}$) to show the $d$-sum is $O(X^\ast
X^{-2+\gep})$. Arguing as above shows the integral is bounded by
$O(X^{-2+2\sigma+\gep})$. (Note: some care is required, as there is
a pole when $w=2$ coming from the trivial zero of $\zeta(s)$ at
$s=-2$. The contribution from the residue here is negligible; we
could also adjust the contour to include a semi-circle around $w=2$
and use the residue theorem.)
\end{proof}

\begin{rek} We sketch an alternate start of the proof of Lemma
\ref{lem:contrRgX}. One difficulty is that $R(g;X)$ is defined as an
integral and there is a pole on the line of integration. We may
write \be \zeta(s) \ = \ (s-1)^{-1} \ +\ \left(\zeta(s) -
(s-1)^{-1}\right). \ee For us $s = 1 - \frac{4\pi i \tau}{\log X}$,
so the first factor is just $-\frac{\log X}{4\pi i \tau}$. As
$g(\tau)$ is an even function, the main term of the integral of this
piece is \bea \intii g(\tau) \frac{e^{-2\pi i \tau}}{2\pi i \tau} \
d\tau & \ = \ & \intii g(\tau) \left(\frac{e^{-2\pi i \tau}}{4\pi i
\tau} - \frac{e^{2\pi i \tau}}{4\pi i \tau}\right) d\tau \nonumber\\
& = & -\intii g(\tau) \frac{\sin(2\pi \tau)}{2\pi \tau}\ d\tau \ = \
-\frac{g(0)}2, \eea where the last equality is a consequence of
$\supp(\hg) \subset (-1,1)$. The other terms from the $(s-1)^{-1}$
factor and the terms from the $\zeta(s) - (s-1)^{-1}$ piece are
analyzed in a similar manner as the terms in the proof of Lemma
\ref{lem:contrRgX}.
\end{rek}

\subsection{Secondary term (of size $1/\log X$) of
$R(g;X)$}\label{sec:secondarytermRgX}

\begin{lem}\label{lem:secondarytermratioslastfactor}
Let $\supp(\hg) \subset (-\sigma, \sigma)$; we do not assume $\sigma
< 1$. Then the $1/\log X$ term in the expansion of $R(g;X)$ is \be
\frac{1 - \frac{\Gamma'(\fof)}{\Gamma(\fof)}
+2\frac{\zeta'(2)}{\zeta(2)} - 2\gamma+ 2\log\pi}{\log X}\ \hg(1).
\ee It is important to note that this piece is \emph{only} present
if the support of $\hg$ exceeds $(-1,1)$ (i.e., if $\sigma > 1$).
\end{lem}

\begin{proof}
We sketch the determination of the main and secondary terms of
$R(g;X)$. We may restrict the integrals to $|\tau| \le \log^{1/4} X$
with negligible error; this will allow us to Taylor expand certain
expressions and maintain good control over the errors. As $g$ is a
Schwartz function, for any $B>0$ we have $g(\tau) \ll
(1+\tau^2)^{-4B}$. The ratio of the Gamma factors is of absolute
value 1, and $A_D(-r;r) = \zeta(2)/\zeta(2-2r) = O(1)$. Thus the
contribution from $|\tau| \ge \log^{1/4} X$ is bounded by \be \ll \
\int_{|\tau|\ge \log^{1/4} X} (1+\tau^2)^{-4B} \cdot
\max\left(\frac{\log X}{\tau}, \frac{\tau^C}{\log^C \tau}\right)
d\tau \ \ll \ (\log X)^{-B} \ee for $B$ sufficiently large.

We use Lemma \ref{lem:partialsumdexpdpiX} to evaluate the $d$-sum in
\eqref{eq:ratiospiecenotinNT} for $|\tau| \le \log^{1/4} X$; the
error term is negligible and may be absorbed into the $O(\log^{-B}
X)$ error. We now Taylor expand the three factors in
\eqref{eq:ratiospiecenotinNT}. The main contribution comes from the
pole of $\zeta$; the other pieces contribute at the $1/\log X$
level.

We first expand the Gamma factors. We have \bea
\frac{\Gamma\left(\fof - \frac{\pi i \tau}{\log
X}\right)}{\Gamma\left(\fof + \frac{\pi i \tau}{\log X}\right)} \ \
\ \ & \ = \ & 1 - \frac{\Gamma'(\fof)}{\Gamma(\fof)} \frac{2\pi
i\tau}{\log X} + O\left(\frac{\tau^2}{\log^2 X}\right). \eea
As $A_D(-r;r) = \zeta(2)/\zeta(2-2r)$, \be A_D\left(-\frac{2\pi i
\tau}{\log X};\frac{2\pi i \tau}{\log X}\right) \ = \ 1 +
2\frac{\zeta'(2)}{\zeta(2)} \frac{2\pi i \tau}{\log X} +
O\left(\frac{\tau^2}{\log^2 X}\right). \ee
Finally we expand the $\zeta$-piece. We have (see \cite{Da}) that
\be \zeta(1+iy) \ = \ \frac1{iy}+\gamma + O(y), \ee where $\gamma$
is Euler's constant. Thus \bea \zeta\left(1 - \frac{4\pi i
\tau}{\log X}\right) & \ = \ & -\frac{\log X}{4\pi i \tau} + \gamma
+ O\left(\frac{\tau}{\log X}\right). \eea

We combine the Taylor expansions for the three pieces (the ratio of
the Gamma factors, the $\zeta$-function and $A_D$), and keep only
the first two terms: \be -\frac{\log X}{4\pi i \tau} + \left[\foh
\frac{\Gamma'(\fof)}{\Gamma(\fof)} -\frac{\zeta'(2)}{\zeta(2)} +
\gamma\right] + O\left(\frac{\tau}{\log X}\right). \ee

Finally, we Taylor expand the $d$-sum, which was evaluated in Lemma
\ref{lem:partialsumdexpdpiX}. We may ignore the error term there
because it is $O(X^{1/2})$. The main term is \bea X^\ast e^{-2\pi i
\left(1-\frac{\log \pi}{\log X}\right) \tau} \left(1 - \frac{2\pi i
\tau}{\log X}\right)^{-1} \ = \   X^\ast e^{-2\pi i
\left(1-\frac{\log \pi}{\log X}\right) \tau} \left(1 + \frac{2\pi i
\tau}{\log X} + O\left(\frac{\tau^2}{\log^2
X}\right)\right).\nonumber\\ \eea

Thus \bea R(g;X) & \ = \ & \frac{-2}{X^\ast\log X} \int_{-\log^{1/4}
X}^{\log X} g(\tau) \cdot X^\ast e^{-2\pi i \left(1-\frac{\log
\pi}{\log^{1/4} X}\right) \tau} \left(1 + \frac{2\pi i \tau}{\log X}
+ O\left(\frac{\tau^2}{\log^2 X}\right)\right) \nonumber\\ & & \ \
\cdot \left[-\frac{\log X}{4\pi i \tau} + \left(\foh
\frac{\Gamma'(\fof)}{\Gamma(\fof)} -\frac{\zeta'(2)}{\zeta(2)} +
\gamma\right) + O\left(\frac{\tau}{\log X}\right)\right]\ d\tau +
O\left(\frac{1}{\log^B X}\right)
\nonumber\\
&=& \frac{2}{\log X} \int_{-\log^{1/4} X}^{\log^{1/4} X} g(\tau)
\cdot e^{-2\pi i \left(1-\frac{\log \pi}{\log X}\right) \tau} \cdot
\left[\frac{\log X}{4\pi i \tau} + \left(\frac{1}2 -\foh
\frac{\Gamma'(\fof)}{\Gamma(\fof)} +\frac{\zeta'(2)}{\zeta(2)} -
\gamma\right) \right]d\tau \nonumber\\ & & \ \ + \
O\left(\frac{1}{\log^{5/4} X}\right). \eea We may write \be e^{-2\pi
i \left(1-\frac{\log \pi}{\log X}\right) \tau} \ = \ e^{-2\pi i
\tau} \cdot \left(1 + \frac{2\pi i \tau\log \pi}{\log
X}+O\left(\frac{\tau^2}{\log^2 X}\right)\right).\ee The effect of
this expansion is to change the $1/\log X$ term above by adding
$\frac{\log \pi}2$.

Because $g$ is a Schwartz function, we may extend the integration to
all $\tau$ and absorb the error into our error term. The main term
is from $(\log X)/4\pi i \tau$; it equals $-g(0)/2$ (see the
analysis in \S\ref{sec:analysisRgX}). The secondary term is easily
evaluated, as it is just the Fourier transform of $g$ at $1$. Thus
\bea R(g;X) & \ = \ & -\frac{g(0)}2 + \frac{1 -
\frac{\Gamma'(\fof)}{\Gamma(\fof)} +2\frac{\zeta'(2)}{\zeta(2)} -
2\gamma+ 2\log\pi}{\log X}\hg(1) + O\left(\frac{1}{\log^{5/4}
X}\right).\nonumber\\ \eea
\end{proof}

\setcounter{equation}{0}

\section{Analysis of the terms from Number Theory}\label{sec:numbthterms}

We now prove Theorem \ref{thm:1ldquadevenfdfromNT}. The starting
point is the Explicit Formula (Theorem \ref{thm:oneldNT}, with each
$d$ an even fundamental discriminant). As the $\log(d/\pi)$ and the
$\Gamma'/\Gamma$ terms already appear in the expansion from the
Ratios Conjecture (Theorem \ref{thm:oneldRatios}), we need only
study the sums of $\chi_d(p)^k$. The analysis splits depending on
whether or not $k$ is even. Set \bea S_{{\rm even}} & \ = \ &
-\frac2{X^\ast}\sum_{d\le X} \sum_{\ell=1}^\infty \sum_p
\frac{\chi_d(p)^2 \log p}{p^\ell \log X}\ \hg\left(2\frac{\log
p^\ell}{\log X}\right) \nonumber\\ S_{{\rm odd}}& \ = \ &
-\frac2{X^\ast}\sum_{d\le X} \sum_{\ell=0}^\infty \sum_p
\frac{\chi_d(p) \log p}{p^{(2\ell+1)/2} \log X}\ \hg\left(\frac{\log
p^{2\ell+1}}{\log X}\right). \eea

Based on our analysis of the terms from the Ratios Conjecture, the
proof of Theorem \ref{thm:1ldquadevenfdfromNT} is completed by the
following lemma.

\begin{lem}\label{lem:contrSevenSodd}
Let $\supp(\hg) \subset (-\sigma, \sigma) \subset (-1,1)$. Then \bea
S_{{\rm even}} & \ = \ & -\frac{g(0)}2 + \frac{2}{\log X}\intii
g(\tau) \frac{\zeta'}{\zeta}\left(1+\frac{4\pi i \tau}{\log
X}\right) d\tau \nonumber\\ & & \ \ \ + \ \frac{2}{\log X}\intii
g(\tau) A_D'\left(\frac{2\pi i \tau}{\log X};\frac{2\pi i \tau}{\log
X}\right) + O(X^{-\foh+\gep}) \nonumber\\ S_{{\rm odd}}& \ = \ &
O(X^{-\frac{1-\sigma}2} \log^6 X). \eea If instead we consider the
family of characters $\chi_{8d}$ for odd, positive square-free $d
\in (0, X)$ ($d$ a fundamental discriminant), then \be S_{{\rm odd}}
\ = \ O(X^{-1/2+\gep} + X^{-(1-\frac32\sigma)+\gep}).\ee
\end{lem}

We prove Lemma \ref{lem:contrSevenSodd} by analyzing $S_{{\rm
even}}$ in \S\ref{sec:evenkcontri} (in Lemmas \ref{lem:ntsumkeven1}
and \ref{lem:ntsumkeven2}) and $S_{{\rm odd}}$ in
\S\ref{sec:oddkcontri} (in Lemmas \ref{lem:proofSodddoesnotcontr},
\ref{lem:proofSodddoesnotcontrgao} and
\ref{lem:proofSodddoesnotcontrgao0toX}).

\subsection{Contribution from $k$ even}\label{sec:evenkcontri}

The contribution from $k$ even from the Explicit Formula is \be
S_{{\rm even}} \ = \ -\frac2{X^\ast}\sum_{d\le X}
\sum_{\ell=1}^\infty \sum_p \frac{\chi_d(p)^2 \log p}{p^\ell \log
X}\ \hg\left(2\frac{\log p^\ell}{\log X}\right), \ee where
$\sum_{d\le X} 1 = X^\ast$, the cardinality of our family. Each
$\chi_d(p)^2 = 1$ except when $p|d$. We replace $\chi_d(p)^2$ with
$1$, and subtract off the contribution from when $p|d$. We find \bea
S_{{\rm even}} & \ = \ & -2\sum_{\ell=1}^\infty \sum_p \frac{\log
p}{p^\ell \log X}\ \hg\left(2\frac{\log p^\ell}{\log X}\right)
\nonumber\\ & & \ \ + \frac2{X^\ast} \sum_{d\le X}
\sum_{\ell=1}^\infty \sum_{p|d} \frac{\log p}{p^\ell \log X}\
\hg\left(2\frac{\log p^\ell}{\log X}\right) \nonumber\\ & = &
S_{{\rm even};1} + S_{{\rm even};2}. \eea In the next subsections we
prove the following lemmas, which completes the analysis of the even
$k$ terms.

\begin{lem}\label{lem:ntsumkeven1} Notation as above, \bea S_{{\rm even};1} & \ =
\ & -\frac{g(0)}2 + \frac{2}{\log X}\intii g(\tau)
\frac{\zeta'}{\zeta}\left(1+\frac{4\pi i \tau}{\log X}\right) d\tau.
\eea
\end{lem}

\begin{lem}\label{lem:ntsumkeven2} Notation as above, \bea
S_{{\rm even};2} & \ = \ & \frac{2}{\log X}\intii
g(\tau) A_D'\left(\frac{2\pi i \tau}{\log X};\frac{2\pi i \tau}{\log
X}\right) + O(X^{-\foh+\gep}). \eea
\end{lem}

\subsubsection{Analysis of $S_{{\rm even};1}$}\ \\

\begin{proof}[Proof of Lemma \ref{lem:ntsumkeven1}]
We have \be S_{{\rm even};1} \ = \ \frac{-2}{\log
X}\sum_{n=1}^\infty \frac{\Lambda(n)}{n} \ \hg\left(2\frac{\log
n}{\log X}\right). \ee We use Perron's formula to re-write $S_{{\rm
even};1}$ as a contour integral. For any $\gep > 0$ set \bea I_1 \ =
\ \ci \int_{\Re(z)=1+\gep} g\left(\frac{(2z-2)\log A}{4\pi i}\right)
\sum_{n=1}^\infty \frac{\Lambda(n)}{n^z}\ dz; \eea we will later
take $A = X^{1/2}$. We write $z = 1+\gep+iy$ and use
\eqref{eq:extendingphi} (replacing $\phi$ with $g$) to write
$g(x+iy)$ in terms of the integral of $\hg(u)$. We have \bea I_1 & \
= \ & \sum_{n=1}^\infty \frac{\Lambda(n)}{n^{1+\gep}} \ci \intii
g\left(\frac{y\log A}{2\pi}-\frac{i\gep\log
A}{2\pi}\right)e^{-iy\log n} idy \nonumber\\ &=& \sum_{n=1}^\infty
\frac{\Lambda(n)}{n^{1+\gep}} \frac1{2\pi} \intii \left[\intii
\left[\hg(u)e^{\gep u \log A}\right] e^{-2\pi i \frac{-y\log A}{2\pi
}u} du \right]e^{-iy\log n}dy.\ \ \ \ \eea We let $h_\gep(u) =
\hg(u) e^{\gep u \log A}$. Note that $h_\gep$ is a smooth, compactly
supported function and $\widehat{\widehat{h_\gep}}(w) = h_\gep(-w)$.
Thus \bea I_1 & \ = \ & \sum_{n=1}^\infty
\frac{\Lambda(n)}{n^{1+\gep}}  \frac1{2\pi} \intii
\widehat{h_\gep}\left(-\frac{y\log A}{2\pi}\right) e^{-iy\log n} dy
\nonumber\\ &=& \sum_{n=1}^\infty \frac{\Lambda(n)}{n^{1+\gep}}
\frac1{2\pi}\intii \widehat{h_\gep}(y) e^{-2\pi i \frac{-y \log
n}{\log A}}\ \frac{2\pi dy}{\log A} \nonumber\\ &=&
\sum_{n=1}^\infty \frac{\Lambda(n)}{n^{1+\gep}} \frac1{\log A}\
\widehat{\widehat{h_\gep}}\left(-\frac{\log n}{\log A}\right)
\nonumber\\
&=& \sum_{n=1}^\infty \frac{\Lambda(n)}{n^{1+\gep}}  \frac1{\log A}\
\hg\left(\frac{\log n}{\log A}\right) e^{\gep \log n} \nonumber\\
&=& \frac1{\log A}\sum_{n=1}^\infty \frac{\Lambda(n)}{n}\
\hg\left(\frac{\log n}{\log A}\right). \eea By taking $A = X^{1/2}$
we find \be S_{{\rm even};1} \ = \ \frac{-2}{\log
X}\sum_{n=1}^\infty \frac{\Lambda(n)}{n} \ \hg\left(2\frac{\log
n}{\log X}\right) \ = \ -I_1.\ee

We now re-write $I_1$ by shifting contours; we will not pass any
poles as we shift. For each $\delta > 0$ we consider the contour
made up of three pieces: $(1-i\infty,1-i\delta]$, $C_\delta$, and
$[1-i\delta,1+i\infty)$, where $C_\delta = \{z: z-1 = \delta
e^{i\theta}, \theta \in [-\pi/2,\pi/2]\}$ is the semi-circle going
counter-clockwise from $1-i\delta$ to $1+i\delta$. By Cauchy's
residue theorem, we may shift the contour in $I_1$ from $\Re(z) =
1+\gep$ to the three curves above. Noting that $\sum_n \Lambda(n)
n^{-z} = -\zeta'(z)/\zeta(z)$, we find that \bea I_1 &\ = \ &
\ci\left[\int_{1-i\infty}^{1-i\delta} + \int_{C_\delta} +
\int_{1+i\delta}^{1+i\infty} g\left(\frac{(2z-2)\log A}{4\pi
i}\right) \frac{-\zeta'(z)}{\zeta(z)}\ dz\right].\eea The integral
over $C_\delta$ is easily evaluated. As $\zeta(s)$ has a pole at
$s=1$, it is just half the residue of $g\left(\frac{(2z-2)\log
A}{4\pi i}\right)$ (the minus sign in front of $\zeta'(z)/\zeta(z)$
cancels the minus sign from the pole). Thus the $C_\delta$ piece is
$g(0)/2$. We now take the limit as $\delta \to 0$: \be I_1 \ = \
\frac{g(0)}2 - \lim_{\delta \to 0} \frac1{2\pi}
\left[\int_{-\infty}^{-\delta} + \int_{\delta}^\infty
g\left(\frac{y\log A}{2\pi}\right) \
\frac{\zeta'(1+iy)}{\zeta(1+iy)}\ dy\right].\ee As $g$ is an even
Schwartz function, the limit of the integral above is well-defined
(for large $y$ this follows from the decay of $g$, while for small
$y$ it follows from the fact that $\zeta'(1+iy)/\zeta(1+iy)$ has a
simple pole at $y=0$ and $g$ is even). We again take $A=X^{1/2}$,
and change variables to $\tau = \frac{y\log A}{2\pi} = \frac{y\log
X}{4\pi}$. Thus \be I_1 \ = \ \frac{g(0)}2 - \frac{2}{\log X}\intii
g(\tau) \frac{\zeta'}{\zeta}\left(1+\frac{4\pi i \tau}{\log
X}\right) d\tau, \ee which completes the proof of Lemma
\ref{lem:ntsumkeven1}. \end{proof}

\subsubsection{Analysis of $S_{{\rm even};2}$}\ \\

\begin{proof}[Proof of Lemma \ref{lem:ntsumkeven2}]
Recall \be S_{{\rm even};2} \ = \ \frac2{X^\ast} \sum_{d\le X}
\sum_{\ell=1}^\infty \sum_{p|d} \frac{\log p}{p^\ell \log X}\
\hg\left(2\frac{\log p^\ell}{\log X}\right). \ee We may restrict the
prime sum to $p \le X^{1/2}$ at a cost of $O(\log\log X / X)$. We
sketch the proof of this claim. Since $\hg$ has finite support, $p
\le X^\sigma$ and thus the $p$-sum is finite. Since $d \le X$ and $p
\ge X^{1/2}$, there are at most 2 primes which divide a given $d$.
Thus \bea  \frac2{X^\ast} \sum_{d\le X} \sum_{\ell=1}^\infty \sum_{p
= X^{1/2} \atop p|d}^{X^\sigma} \frac{\log p}{p^\ell \log X}\
\hg\left(2\frac{\log p^\ell}{\log X}\right) & \ \ll \ &
\frac1{X^\ast} \sum_{\ell=1}^\infty \sum_{p = X^{1/2}}^{X^\sigma}
\frac{1}{p^\ell} \sum_{d\le X \atop p|d} 1 \nonumber\\ & \ \ll\ &
\frac1{X^\ast} \sum_{p > X^{1/2}}^{X^\sigma} \frac{2}{p} \ \ll \
\frac{\log \log X}{X}.\ \ \ \ \ \eea

In Lemma \ref{lem:numdwithpdivided} we show that \be X^\ast \ = \
\frac{3}{\pi^2}X + O(X^{1/2}) \ee and that for $p \le X^{1/2}$ we
have \be \sum_{d \le X \atop p|d} 1 \ = \ \frac{X^\ast}{p+1} +
O(X^{1/2}). \ee Using these facts we may complete the analysis of
$S_{{\rm even};2}$: \bea S_{{\rm even};2} &\ = \ & \frac2{X^\ast}
\sum_{d\le X} \sum_{\ell=1}^\infty \sum_{p \le X^{1/2} \atop p|d}
\frac{\log p}{p^\ell \log X}\ \hg\left(2\frac{\log p^\ell}{\log
X}\right) + O\left(\frac{\log\log X}{X}\right)\nonumber\\ & = &
\frac2{X^\ast} \sum_{\ell=1}^\infty \sum_{p \le X^{1/2}} \frac{\log
p}{p^\ell \log X}\ \hg\left(2\frac{\log p^\ell}{\log X}\right)
\sum_{d\le X,\ p|d} 1 + O\left(\frac{\log\log
X}{X}\right)\nonumber\\ &=& 2 \sum_{\ell=1}^\infty \sum_{p \le
X^{1/2}} \frac{\log p}{p^\ell \log X}\cdot \frac{1}{p+1}\
\hg\left(2\frac{\log p^\ell}{\log X}\right) \nonumber\\ & & \ \ \ +\
O\left(\frac{X^{1/2}}{X} \sum_{\ell=1}^\infty \sum_{p \le X^{1/2}}
\frac1{p^\ell} + \frac{\log\log X}{X}\right) \nonumber\\ &=& 2
\sum_{\ell=1}^\infty \sum_{p \le X^{1/2}} \frac{\log p}{p^\ell \log
X}\cdot \frac{1}{p+1}\ \hg\left(2\frac{\log p^\ell}{\log X}\right) +
O(X^{-\foh+\gep}).\eea

We re-write $\hg(2\log p^\ell/\log X)$ by expanding the Fourier
transform. \begin{align} & S_{{\rm even};2}  \ = \  2
\sum_{\ell=1}^\infty \sum_{p \le X^{1/2}} \frac{\log p}{(p+1)p^\ell
\log X} \intii g(\tau) e^{-2\pi i\tau \cdot 2\log p^\ell / \log X}
d\tau +
O(X^{-\foh+\gep})\nonumber\\
=\ & 2 \sum_{p \le X^{1/2}} \frac{\log p}{(p+1)\log X}\intii g(\tau)
\sum_{\ell=1}^\infty p^{-\ell} \cdot p^{-2\pi i\tau \cdot 2\ell/\log
X} d\tau + O(X^{-\foh+\gep}) \nonumber\\ = \ & 2\sum_{p\le X^{1/2}}
\frac{\log p}{(p+1)\log X}\intii g(\tau) p^{-(1+2\cdot \frac{2\pi
i\tau}{\log X})} \left(1 - p^{-(1+2\cdot\frac{2\pi i\tau}{\log
X})}\right)^{-1}d\tau + O(X^{-\foh+\gep}). \nonumber\\ \end{align}
We may extend the $p$-sum to be over all primes at a cost of
$O(X^{-1/2+\gep})$; this is because the summands are $O(\log p /
p^2)$ and $g$ is Schwartz. Recalling the definition of $A_D'(r;r)$
in \eqref{eq:defnADAD'}, we see that the resulting $p$-sum is just
$A_D'(2\pi i \tau/\log X; 2\pi i \tau/\log X)$; this completes the
proof of Lemma \ref{lem:ntsumkeven2}. \end{proof}


\subsection{Contribution from $k$ odd}\label{sec:oddkcontri}

As $k$ is odd, $\chi_d(p)^k=\chi_d(p)$. Thus we must analyze the sum
\be\label{eq:expansionSodd} S_{{\rm odd}} \ = \
-\frac2{X^\ast}\sum_{d\le X} \sum_{\ell=0}^\infty \sum_p
\frac{\chi_d(p) \log p}{p^{(2\ell+1)/2} \log X}\ \hg\left(\frac{\log
p^{2\ell+1}}{\log X}\right). \ee If ${\rm supp}(\hg) \subset (-1,
1)$, Rubinstein \cite{Rub1} showed (by applying Jutila's bound
\cite{Ju1,Ju2,Ju3} for quadratic character sums) that if our family
is all discriminants then $S_{{\rm odd}} = O(X^{-\gep/2})$. In his
dissertation Gao \cite{Gao} extended these results to show that the
odd terms do not contribute to the main term provided that ${\rm
supp}(\hg) \subset (-2, 2)$. His analysis proceeds by using Poisson
summation to convert long character sums to shorter ones. We shall
analyze $S_{{\rm odd}}$ using both methods: Jutila's bound gives a
self-contained presentation, but a much weaker result; the Poisson
summation approach gives a better bound but requires a careful
book-keeping of many of Gao's lemmas (as well as an improvement of
one of his estimates).

\subsubsection{Analyzing $S_{{\rm odd}}$ with Jutila's
bound}\label{sec:analyzingSoddJutila}

\begin{lem}\label{lem:proofSodddoesnotcontr} Let $\supp(\hg) \subset
(-\sigma, \sigma)$. Then $S_{{\rm odd}} = O(X^{-\frac{1-\sigma}2}
\log^6 X)$.
\end{lem}

\begin{proof} Jutila's bound (see (3.4) of \cite{Ju3}) is \be \sum_{1 < n \le N
\atop n\ {\rm non-square}} \ \left| \sum_{0 < d \le X \atop d\ {\rm
fund.\ disc.}}\ \chi_d(n)\right|^2 \ \ll \ N X \log^{10} N \ee (note
the $d$-sum is over even fundamental discriminants at most $X$). As
$2\ell+1$ is odd, $p^{2\ell+1}$ is never a square. Thus Jutila's
bound gives \be \left(\sum_{\ell=0}^\infty \sum_{p^{(2\ell+1)/2} \le
X^\sigma} \left| \sum_{d \le X} \chi_d(p)\right|^2 \right)^{1/2} \
\ll \ X^{\frac{1+\sigma}2} \log^5 X. \ee Recall \bea S_{{\rm odd}}
&\ =\ & -\frac2{X^\ast} \sum_{\ell=0}^\infty \sum_p \frac{\log
p}{p^{(2\ell+1)/2} \log X}\ \hg\left(\frac{\log p^{2\ell+1}}{\log
X}\right) \sum_{d\le X} \chi_d(p). \eea We apply Cauchy-Schwartz,
and find \bea |S_{{\rm odd}}| &\ \le \ &
\frac2{X^\ast}\left(\sum_{\ell=0}^\infty \sum_{p^{2\ell+1} \le
X^\sigma} \left|\frac{\log p}{p^{(2\ell+1)/2} \log X}\
\hg\left(\frac{\log p^{2\ell+1}}{\log
X}\right)\right|^2\right)^{1/2} \nonumber\\ & & \ \ \cdot \
\left(\sum_{\ell=0}^\infty \sum_{p^{2\ell+1} \le X^\sigma}
\left|\sum_{d\le X} \chi_d(p)\right|^2\right)^{1/2} \nonumber\\ &\ll
& \frac2{X^\ast}\left(\sum_{n \le X^\sigma} \frac1{n}\right)^{1/2}
\cdot X^{\frac{1+\sigma}2} \log^5 X \nonumber\\ &\ll &
X^{-\frac{1-\sigma}2}  \log^6 X; \eea thus there is a power savings
if $\sigma < 1$.
\end{proof}

\subsubsection{Analyzing $S_{{\rm odd}}$ with Poisson
Summation}\label{sec:analyzingSoddPoissonSum}\ \\

Gao analyzes the contribution from $S_{{\rm odd}}$ by applying
Poisson summation to the character sums. The computations are
simplified if the character $\chi_2(n) = \jsgeneral{2}{n}$ is not
present. He therefore studies the family of odd, positive
square-free $d$ (where $d$ is a fundamental discriminant). His
family is \be\label{eq:gaofamily8d} \{8d:\ X < d \le 2X, \ d \ {\rm
an\ odd\ square-free\ fundamental\ discriminant}\}; \ee we discuss
in Lemma \ref{lem:proofSodddoesnotcontrgao0toX} how to easily modify
the arguments to handle the related family with $0 < d \le X$. The
calculation of the terms from the Ratios Conjecture proceeds
similarly (the only modification is to $X^\ast$, which also leads to
a trivial modification of Lemma \ref{lem:partialsumdexpdpiX} which
does not change any terms larger than $O(X^{-1/2+\gep})$ if
$\supp(\hg) \subset (-1/3, 1/3)$), as does the contribution from
$\chi(p)^k$ with $k$ even. We are left with bounding the
contribution from $S_{{\rm odd}}$. The following lemma shows that we
can improve on the estimate obtained by applying Jutila's bound.

\begin{lem}\label{lem:proofSodddoesnotcontrgao} Let $\supp(\hg) \subset
(-\sigma, \sigma) \subset (-1,1)$. Then for the family given in
\eqref{eq:gaofamily8d}, $S_{{\rm odd}} =
O(X^{-\foh+\gep}+X^{-(1-\frac32\sigma)+\gep})$. In particular, if
$\sigma < 1/3$ then $S_{{\rm odd}} = O(X^{-1/2+\gep})$.
\end{lem}

\begin{proof} Gao is only concerned with main terms for the
$n$-level density (for any $n$) for all sums. As we only care about
$S_{{\rm odd}}$ for the $1$-level density, many of his terms are not
present. We highlight the arguments. We concentrate on the $\ell =
0$ term in \eqref{eq:expansionSodd} (the other $\ell \ll \log X$
terms are handled similarly, and the finite support of $\hg$ implies
that $S_{{\rm odd}}(\ell) = 0$ for $\ell \gg \log X$): \bea S_{{\rm
odd}} \ = \ -\frac2{X^\ast}\sum_{d\le X} \sum_{\ell=0}^\infty \sum_p
\frac{\chi_d(p) \log p}{p^{(2\ell+1)/2} \log X}\ \hg\left(\frac{\log
p^{2\ell+1}}{\log X}\right) \ = \ \sum_{\ell=0}^\infty S_{{\rm
odd}}(\ell). \eea

Let $Y = X^\sigma$, where $\supp(\hg) \subset (-\sigma, \sigma)$.
Our sum $S_{{\rm odd}}(0)$ is $S(X,Y,\hg)$ in Gao's thesis: \be
S(X,Y,\hg) \ = \ \sum_{X < d < 2X \atop (2,d) = 1} \mu(d)^2 \sum_{p
< Y} \frac{\log p}{\sqrt{p}} \chi_{8d}(p) \hg\left(\frac{\log
p}{\log X}\right). \ee

Let $\Phi$ be a smooth function supported on $(1,2)$ such that
$\Phi(t) = 1$ for $t \in (1 + U^{-1},2 - U^{-1})$ and $\Phi^{(j)}(t)
\ll_j U^j$ for all $j \ge 0$. We show that $S(X,Y,\hg)$ is well
approximated by the smoothed sum $S(X,Y,\hg,\Phi)$, where \be
S(X,Y,\hg,\Phi) \ = \ \sum_{(d,2) = 1} \mu(d)^2 \sum_{p < Y}
\frac{\log p}{\sqrt{p}} \chi_{8d}(p) \hg\left(\frac{\log p}{\log
X}\right) \Phi\left(\frac{d}{X}\right). \ee To see this, note the
difference between the two involves summing $d \in (X, X + X/U)$ and
$d \in (2X - X/U, 2X)$. We trivially bound the prime sum for each
fixed $d$ by $\log^7 X$ (see Proposition III.1 of \cite{Gao}). As
there are $O(X/U)$ choices of $d$ and $\Phi(d/X) \ll 1$, we have \be
S(X,Y,\hg) - S(X,Y,\hg,\Phi) \ \ll \ \frac{X\log^7 X}{U}. \ee We
will take $U = \sqrt{X}$. Thus upon dividing by $X^\ast \gg X$ (the
cardinality of the family), this difference is $O(X^{-1/2+\gep})$.
The proof is completed by bounding $S(X,Y,\hg,\Phi)$.

To analyze $S(X,Y,\hg,\Phi)$, we write it as $S_M(X,Y,\hg,\Phi) +
S_R(X,Y,\hg,\Phi)$, with \bea S_M(X,Y,\hg,\Phi) & \ = \ &
\sum_{(d,2)=1} M_Z(d) \sum_{p < Y} \frac{\log p}{\sqrt{p}}
\chi_{8d}(p) \hg\left(\frac{\log p}{\log X}\right)
\Phi\left(\frac{d}{X}\right) \nonumber\\ S_R(X,Y,\hg,\Phi) & = &
\sum_{(d,2)=1} R_Z(d) \sum_{p < Y} \frac{\log p}{\sqrt{p}}
\chi_{8d}(p) \hg\left(\frac{\log p}{\log X}\right)
\Phi\left(\frac{d}{X}\right), \eea where \bea \mu(d)^2 & \ = \ &
M_Z(d) + R_Z(d) \nonumber\\ & & M_Z(d) \ = \ \sum_{\ell^2|d \atop
\ell \le Z} \mu(\ell), \ \ \ \ R_Z(d) \ = \ \sum_{\ell^2|d \atop
\ell > Z} \mu(\ell); \eea here $Z$ is a parameter to be chosen
later, and $S_M(X,Y,\hg,\Phi)$ will be the main term (for a general
$n$-level density sum) and $S_R(X,Y,\hg,\Phi)$ the error term. In
our situation, both will be small.

In Lemma III.2 of \cite{Gao}, Gao proves that $S_R(X,Y,\hg,\Phi) \ll
(X \log^3 X) / Z$. We haven't divided any of our sums by the
cardinality of the family (which is of size $X$). Thus for this term
to yield contributions of size $X^{-1/2 + \gep}$, we need $Z \ge
X^{1/2}$.

We now analyze $S_M(X,Y,\hg,\Phi)$. Applying Poisson summation we
convert long character sums to short ones. We need certain
Gauss-type sums: \be \left(\frac{1+i}2 + \jsgeneral{-1}{k}
\frac{1-i}2\right) G_m(k)  \ = \ \sum_{a \bmod k} \jsgeneral{a}{k}
e^{2\pi i a m / k}. \ee For a Schwartz function $F$ let \be
\widetilde{F}(\xi) \ = \ \frac{1+i}2 \widehat{F}(\xi) +
\frac{1-i}2\widehat{F}(-\xi). \ee Using Lemma 2.6 of \cite{So}, we
have (see page 32 of \cite{Gao}) \bea\label{eq:expansionSMXYhgphi}
S_M(X,Y,\hg,\Phi) & \ = \ & \frac{X}2 \sum_{2<p<Y} \frac{\log
p}{p^{3/2}} \hg\left(\frac{\log p}{\log X}\right) \nonumber\\ & & \
\ \ \ \cdot \sum_{\alpha \le Z \atop (\alpha,2p) = 1}
\frac{\mu(\alpha)}\alpha \sum_{m=0}^\infty (-1)^m G_m(p)
\widetilde{\Phi}\left(\frac{mX}{2\alpha^2 p}\right). \eea We follow
the arguments in Chapter 3 of \cite{Gao}. The $m=0$ term is analyzed
in \S3.3 for the general $n$-level density calculations. It is zero
if $n$ is odd, and we do not need to worry about this error term
(thus we do not see the error terms of size $X \log^{n-1} X$ or
$(X\log^n X)/Z$ which appear in his later estimates). In \S3.4 he
analyzes the contributions from the non-square $m$ in
\eqref{eq:expansionSMXYhgphi}. In his notation, we have $k=1$, $k_2
= 0$, $k_1 = 0$, $\alpha_1 = 1$ and $\alpha_0 = 0$, and these terms'
contribution is $\ll (U^2 Z \sqrt{Y} \log^7 X)/X$ (remember we
haven't divided by the cardinality of the family, which is of order
$X$). This is too large for our purposes (we have seen that we must
take $U = Z = \sqrt{X}$ and $Y = X^\sigma$). We perform a more
careful analysis of these terms in Appendix
\ref{sec:improvingGaomnotsquare}, and bound these terms'
contribution by \be \frac{UZ\sqrt{Y}\log^7 X}{X} +
\frac{UZY^{3/2}\log^4 X}{X} + \frac{Z^3 U^2 Y^{7/2} \log^4
X}{X^{4018-2\gep}}. \ee

Lastly, we must analyze the contribution from $m$ a square in
\eqref{eq:expansionSMXYhgphi}. From Lemma III.3 of \cite{Gao} we
have that $G_m(p) = 0$ if $p|m$. If $p\notdiv m$ and $m$ is a
square, then $G_m(p) = \sqrt{p}$. Arguing as in \cite{Gao}, we are
left with \bea \sum_{p < Y \atop (p,2)=1} \frac{\log p}{p}
\hg\left(\frac{\log p}{\log X}\right) \sum_{\alpha \le Z \atop
(\alpha,2p) = 1} \frac{\mu(\alpha)}{\alpha^2}\left[\sum_{m=1}^\infty
(-1)^m \widetilde{\Phi}\left(\frac{m^2 X}{2\alpha^2 p}\right) -
\sum_{\widetilde{m}=1}^\infty (-1)^{\widetilde{m}}
\widetilde{\Phi}\left(\frac{p^2 \widetilde{m}^2 X}{2\alpha^2
p}\right)\right].\nonumber\\ \eea If we assume $\supp(\hg) \subset
(-1,1)$, then arguing as on page 41 of \cite{Gao} we find the
$m$-sum above is $\ll \alpha \sqrt{p/X}$, which leads to a
contribution $\ll \sqrt{Y/X} \log X \log Z$; the $\widetilde{m}$-sum
is $\ll \alpha/\sqrt{pX}$ and is thus dominated by the contribution
from the $m$-sum.

Collecting all our bounds, we see a careful book-keeping leads to
smaller errors than in \S3.6 of \cite{Gao} (this is because (1) many
of the error terms only arise from $n$-level density sums with $n$
even, where there are main terms and (2) we did a more careful
analysis of some of the errors). We find that
\be\label{eq:SXYhgPhiexpansionfourterms} S(X,Y,\hg,\Phi) \ \ll \
\frac{X\log^3 X}{Z} + \frac{UZ\sqrt{Y}\log^7 X}{X}  +
\frac{UZY^{3/2}\log^4 X}{X} + \frac{\sqrt{Y} \log X \log
Z}{\sqrt{X}}. \ee We divide this by $X^\ast \gg X$ (the cardinality
of the family). By choosing $Z = X^{1/2}$, $Y = X^\sigma$ with
$\sigma < 1$, and $U = \sqrt{X}$ (remember we need such a large $U$
to handle the error from smoothing the $d$-sum, i.e., showing
$|S(X,Y,\hg) - S(X,Y,\hg,\Phi)|/X$ $\ll$ $X^{-1/2+\gep}$), we find
\be S(X,Y,\hg,\Phi)/X \ \ll \ X^{-1/2+\gep} +
X^{-(1-\frac32\sigma)+\gep},\ee which yields \be S_{{\rm odd}} \ \ll
\ X^{-1/2+\gep} + X^{-(1-\frac32\sigma)+\gep}.\ee Note that if
$\sigma < 1/3$ then $S_{{\rm odd}} \ll X^{-1/2+\gep}$.
\end{proof}

\begin{lem}\label{lem:proofSodddoesnotcontrgao0toX} Let $\supp(\hg) \subset
(-\sigma, \sigma) \subset (-1,1)$. Then for the family
\be\label{eq:gaofamily8d0toX} \{8d:\ 0 < d \le X, \ d \ {\rm an\
odd\ square-free\ fundamental\ discriminant}\} \ee we have $S_{{\rm
odd}} = O(X^{-1/2+\gep} + X^{-(1-\frac32\sigma)+\gep})$. In
particular, if $\sigma < 1/3$ then $S_{{\rm odd}} =
O(X^{-1/2+\gep})$.
\end{lem}

\begin{proof} As the calculation is standard, we merely sketch the
argument. We write \be (0, X] \ = \ \bigcup_{i=1}^{\log_2 X}
\left(\frac{2X}{2^{i+1}},\ \frac{2X}{2^i}\right]. \ee Let $X_i =
X/2^i$. For each $i$, in Lemma \ref{lem:proofSodddoesnotcontrgao} we
replace most of the $X$'s with $X_i$, $U$ with $U/\sqrt{2^i}$, $Z$
with $Z/\sqrt{2^i}$; the $X$'s we don't replace are the cardinality
of the family (which we divide by in the end) and the $\log X$ which
occurs when we evaluate the test function $\hg$ at $\log p/\log X$.
We do not change $Y$, which controls the bounds for the prime sum.
As we do not have any main terms, there is no loss in scaling the
prime sums by $\log X$ instead of $\log X_i$. We do not use much
about the test function $\hg$ in our estimates. All we use is that
the prime sums are restricted to $p < Y$, and therefore we will
still have bounds of $Y$ (to various powers) for our sums.

We now finish the book-keeping. Expressions such as $UZ/X$ in
\eqref{eq:SXYhgPhiexpansionfourterms} are still $O(1)$, and
expressions such as $X/U$ and $X/Z$ are now smaller. When we divide
by the cardinality of the family we still have terms such as
$Y^{3/2}/X$, and thus the support requirements are unchanged (i.e.,
$S_{{\rm odd}} \ll X^{-1/2+\gep} + X^{-(1-\frac32\sigma)+\gep}$).
\end{proof}




\appendix

\setcounter{equation}{0}

\section{The Explicit Formula}\label{sec:explicitformula}

We quickly review some needed facts about Dirichlet characters; see
\cite{Da} for details. Let $\chi_d$ be a primitive quadratic
Dirichlet character of modulus $|d|$. Let $c(d,\chi_d)$ be the Gauss
sum \be c(d,\chi_d) \ = \ \sum_{k=1}^{d-1} \chi_d(k) e^{2\pi i k/d},
\ee which is of modulus $\sqrt{d}$. Let \be L(s,\chi_d) \ = \
\prod_p \left(1 -\chi_d(p)p^{-s}\right)^{-1} \ee be the $L$-function
attached to $\chi_d$; the completed $L$-function is \be
\Lambda(s,\chi_d) \ = \ \pi^{-(s+a)/2} \gam{s+a}{2} d^{-(s+a)/2}
L(s,\chi_d) \ = \ (-1)^a \frac{c(d,\chi_d)}{\sqrt{d}}
\Lambda(1-s,\overline{\chi_d}), \ee where \be \twocase{a \ = \
a(\chi_d) \ = \ }{0}{if $\chi_d(-1) = 1$}{1}{if $\chi_d(-1) =
-1$.}\ee

We write the zeros of $\Lambda(s,\chi_d)$ as $\foh+i\gamma$; if we
assume GRH then $\gamma\in\R$. Let $\phi$ be an even Schwartz
function and $\hphi$ be its Fourier transform ($\hphi(\xi) = \int
\phi(x) e^{-2\pi i x \xi}dx$); we often assume ${\rm supp}(\hphi)
\subset (-\sigma, \sigma)$ for some $\sigma < \infty$. We set \be
H(s) \ = \ \phi\left(\frac{s - \foh}{i}\right). \ee While $H(s)$ is
initially define only when $\Re(s) = 1/2$, because of the compact
support of $\hphi$ we may extend it to all of $\C$:
\bea\label{eq:extendingphi}
\phi(x) &\ = \ & \intii \hphi(\xi) e^{2\pi i x\xi} d\xi \nonumber\\
\phi(x+iy) &=& \intii \hphi(\xi) e^{2\pi i(x+iy)\xi} d\xi
\nonumber\\ H(x+iy) &=& \intii \left[\hphi(\xi)e^{2\pi(x -
\foh)}\right] \cdot e^{2\pi i y \xi} d\xi. \eea Note that $H(x+iy)$
is rapidly decreasing in $y$ (for a fixed $x$ it is the Fourier
transform of a nice function, and thus the claim follows from the
Riemann-Lebesgue lemma). We now derive the Explicit Formula for
quadratic characters; note the functional equation will always be
even. We follow the argument given in \cite{RS}.

\begin{proof}[Proof of the Explicit Formula, Theorem \ref{thm:oneldNT}] We have
 \bea\label{eq:logderivDir} \Lambda(s,\chi
 ) &\ = \ &
\pi^{-(s+a)/2} \gam{s+a}{2} d^{(s+a)/2} L(s,\chi_d) \ = \
\Lambda(1-s,\chi_d) \nonumber\\
\frac{\Lambda'(s,\chi_d)}{\Lambda(s,\chi_d)}  &=& -\frac{\log\pi}2 +
\foh \frac{\Gamma'}{\Gamma}\left(\frac{s+a}{2}\right) + \frac{\log
d}{2} + \frac{L'(s,\chi_d)}{L(s,\chi_d)} \nonumber\\
\frac{L'(s,\chi_d)}{L(s,\chi_d)} & = & -\sum_p \frac{\chi_d(p)\log
p}{1 - \chi_d(p)p^{-s}} \ = \ -\sum_{k=1}^\infty \sum_p
\frac{\chi_d(p)^k\log p}{p^{ks}}.\eea We will \emph{not} approximate
any terms; we are keeping all lower order terms to facilitate
comparison with the $L$-functions Ratios Conjecture. We set \be I \
= \ \ci \int_{\Re(s) = 3/2}
\frac{\Lambda'(s,\chi_d)}{\Lambda(s,\chi_d)} H(s)ds. \ee We shift
the contour to $\Re(s) = -1/2$. We pick up contributions from the
zeros and poles of $\Lambda(s,\chi_d)$. As $\chi_d$ is not the
principal character, there is no pole from $L(s,\chi_d)$. There is
also no need to worry about a zero or pole from the Gamma factor
$\gam{s+a}{2}$ as $L(1,\chi_d) \neq 0$. Thus the only contribution
is from the zeros of $\Lambda(s,\chi_d)$; the residue at a zero
$\foh + i\gamma$ is $\phi(\gamma)$. Therefore \bea I & \ = \ &
\sum_\gamma \phi(\gamma) + \ci \int_{\Re(s) = -1/2}
\frac{\Lambda'(s,\chi_d)}{\Lambda(s,\chi_d)} H(s)ds. \eea As
$\Lambda(1-s,\chi_d) = \Lambda(s,\chi_d)$, $-\Lambda'(1-s,\chi_d) =
\Lambda(s,\chi_d)$ and \bea I & \ = \ & \sum_\gamma \phi(\gamma) -
\ci \int_{\Re(s)=-1/2}
\frac{\Lambda'(1-s,\chi_d)}{\Lambda(1-s,\chi_d)} H(s)ds. \eea We
change variables (replacing $s$ with $1-s$), and then use the
functional equation: \bea I & \ = \ & \sum_\gamma \phi(\gamma) - \ci
\int_{\Re(s) = 3/2} \frac{\Lambda'(s,\chi_d)}{\Lambda(s,\chi_d)}
H(1-s)ds. \eea Recalling the definition of $I$ gives \be \sum_\gamma
\phi(\gamma)\ =\ \ci \int_{\Re(s)=3/2}
\frac{\Lambda'(s,\chi_d)}{\Lambda(s,\chi_d)}\left[H(s) +
H(1-s)\right] ds. \ee We expand
$\Lambda'(s,\chi_d)/\Lambda(s,\chi_d)$ and shift the contours of all
terms except $L'(s,\chi_d)/L(s,\chi_d)$ to $\Re(s) = 1/2$ (this is
permissible as we do not pass through any zeros or poles of the
other terms); note that if $s=\foh+iy$ then $H(s) = H(1-s) =
\phi(y)$ ($\phi$ is even). Expanding the logarithmic derivative of
$\Lambda(s,\chi_d)$ gives
\bea \sum_\gamma \phi(\gamma) & \ = \ & \frac1{2\pi} \intii
\left[\log \frac{d}{\pi} +
\frac{\Gamma'}{\Gamma}\left(\frac14+\frac{a}2+\frac{iy}2\right)\right]
\phi(y)dy\nonumber\\ & &\ \ +\ \frac1{2\pi i} \int_{\Re(s)=3/2}
\frac{L'(s,\chi_d)}{L(s,\chi_d)} \cdot \left[H(s) + H(1-s)\right] ds
\nonumber\\ &=& \frac1{2\pi} \intii \left[\log \frac{d}{\pi} + \foh
\frac{\Gamma'}{\Gamma}\left(\frac14+\frac{a}2+\frac{iy}2\right)+\foh
\frac{\Gamma'}{\Gamma}\left(\frac14+\frac{a}2-\frac{iy}2\right)\right]
\phi(y)dy\nonumber\\ & &\ \ +\ \frac1{2\pi i} \int_{\Re(s)=3/2}
\frac{L'(s,\chi_d)}{L(s,\chi_d)} \cdot \left[H(s) + H(1-s)\right]
ds, \eea where the last line follows from the fact that $\phi$ is
even.

We use \eqref{eq:logderivDir} to expand $L'/L$. In the arguments
below we shift the contour to $\Re{s} = 1/2$; this is permissible
because of the compact support of $\hphi$ (see
\eqref{eq:extendingphi}): \bea & & \frac1{2\pi i} \int_{\Re(s)=3/2}
\frac{L'}{L}\left(s+iy\right) \cdot
\left[H\left(s\right)+H\left(1-s\right)\right]dy \nonumber\\  = \ &
& -\frac1{2\pi i} \sum_{k=1}^\infty \sum_p \chi_d(p)^k \log p
\int_{\Re(s)=3/2}
\left[H\left(s\right)+H\left(1-s\right)\right] e^{-ks\log p} dy \nonumber\\
=\ & & -\frac2{2\pi} \sum_{k=1}^\infty \sum_p \frac{\chi_d(p)^k\log
p}{p^{k/2}} \intii \phi(y) e^{-2\pi i y \cdot \frac{\log
p^k}{2\pi}}dy \nonumber\\ = \ & & -\frac2{2\pi}\sum_{k=1}^\infty
\sum_p \frac{\chi_d(p)^k\log p}{p^{k/2}}\  \hphi\left(\frac{\log
p^k}{2\pi}\right).\eea

We therefore find that \bea \sum_\gamma \phi(\gamma) &\ = \ &
\frac1{2\pi}\intii \left[\log \frac{d}{\pi} + \foh
\frac{\Gamma'}{\Gamma}\left(\frac14+\frac{a}2+\frac{iy}2\right)+\foh
\frac{\Gamma'}{\Gamma}\left(\frac14+\frac{a}2-\frac{iy}2\right)\right]
\phi(y)dy\nonumber\\ & &\ \ -\ \frac2{2\pi} \sum_{k=1}^\infty \sum_p
\frac{\chi_d(p)^k\log p}{p^{k/2}}\  \hphi\left(\frac{\log
p^k}{2\pi}\right).\eea

We replace $\phi(x)$ with $g(x) = \phi\left(x \cdot\frac{\log
X}{2\pi}\right)$. A standard computation gives $\hg(\xi) =
\frac{2\pi}{\log X}\ \hphi\left(\xi \cdot \frac{2\pi}{\log
X}\right)$. Summing over $d \in \mathcal{F}(X)$ completes the proof.
\end{proof}

\setcounter{equation}{0}

\section{Sums over fundamental
discriminants}\label{sec:sumsoverfunddiscs}

\begin{lem}\label{lem:numdwithpdivided} Let $d$ denote an even
fundamental discriminant at most $X$, and set $X^\ast = \sum_{d \le
X} 1$. Then \be X^\ast   \ = \ \frac{3}{\pi^2}X + O(X^{1/2}) \ee and
for $p \le X^{1/2}$ we have \be \sum_{d \le X \atop p|d} 1 \ = \
\frac{X^\ast}{p+1} + O(X^{1/2}). \ee
\end{lem}

\begin{proof} We first prove the claim for $X^\ast$, and then
indicate how to modify the proof when $p|d$. We could show this by
recognizing certain products as ratios of zeta functions or by using
a Tauberian theorem; instead we shall give a straightforward proof
suggested to us by Tim Browning (see also \cite{OS1}).

We first assume that $d \equiv 1 \bmod 4$, so we are considering
even fundamental discriminants $\{d \le X: d \equiv 1 \bmod 4,
\mu(d)^2=1\}$; it is trivial to modify the arguments below for $d$
such that $d/4 \equiv 2$ or $3$ modulo $4$ and $\mu(d/4)^2=1$. Let
$\chi_4(n)$ be the non-trivial character modulo 4: $\chi_4(2m) = 0$
and \be \twocase{\chi_4(n) \ = \ }{1}{if $n\equiv 1 \bmod 4$}{0}{if
$n \equiv 3 \bmod 4$.} \ee We have \bea\label{eq:startSXanalysis}
S(X) & \ = \ & \sum_{d\le X \atop \mu(d)^2=1,\ d\equiv 1 \bmod 4} 1
\nonumber\\ & \ = \ & \sum_{d \le
X \atop 2 \notdiv d} \mu(d)^2 \cdot \frac{1 + \chi_4(d)}2 \nonumber\\
&=& \foh \sum_{d\le X \atop 2 \notdiv d} \mu(d)^2 + \foh \sum_{d\le
X} \mu(d)^2 \chi_4(d) \ = \ S_1(X) + S_2(X).\eea By M\"obius
inversion \be \twocase{\sum_{m^2|d} \mu(m) \ = \ }{1}{if $d$ is
square-free}{0}{otherwise.} \ee Thus \bea S_1(X) & \ = \ & \foh
\sum_{d \le X \atop 2 \notdiv d} \sum_{m^2|d} \mu(m) \nonumber\\ &=&
\foh \sum_{m \le X^{1/2}\atop 2\notdiv m} \mu(m) \cdot \sum_{d\ \le\
X/m^2 \atop 2\notdiv d} 1 \nonumber\\ &=& \foh \sum_{m\le X^{1/2}
\atop 2\notdiv m} \mu(m)\left(\frac{X}{2m^2} + O(1)\right)
\nonumber\\ &=& \frac{X}{4} \sum_{m=1 \atop 2\notdiv m}^\infty
\frac{\mu(m)}{m^2} + O(X^{1/2}) \nonumber\\ &=& \fof
\frac{6}{\zeta(2)} \cdot \left(1 - \frac1{2^2}\right)^{-1}\cdot X +
O(X^{1/2}) \nonumber\\ &=& \frac{2}{\pi^2}X + O(X^{1/2}) \eea
(because we are missing the factor corresponding to $2$ in
$1/\zeta(2)$ above). Arguing in a similar manner shows $S_2(X) =
O(X^{1/2})$; this is due to the presence of $\chi_4$, giving us \be
S_2(X) \ = \ \foh \sum_{m \le X^{1/2}} \chi_4(m^2)\mu(m) \sum_{d\le
X/m^2} \chi_4(d) \ \ll \ X^{1/2} \ee (because we are summing
$\chi_4$ at consecutive integers, and thus this sum is at most 1). A
similar analysis shows that the number of even fundamental
discriminants $d\le X$ with $d/4 \equiv 2$ or $3$ modulo $4$ is
$X/\pi^2 + O(X^{1/2})$. Thus \be \sum_{d\le X \atop d\ {\rm an\
even\ fund.\ disc.}}1 \ = \ X^\ast \ = \
\frac{3}{\pi^2}X+O(X^{1/2}).\ee

We may trivially modify the above calculations to determine the
number of even fundamental discriminants $d\le X$ with $p|d$ for a
fixed prime $p$. We first assume $p\equiv 1 \bmod 4$. In
\eqref{eq:startSXanalysis} we replace $\mu(d)^2$ with $\mu(pd)^2$,
$d \le X$ with $d \le X/p$, $2\notdiv d$ and $(2p,d)=1$. These imply
that $d \le X$, $p|d$ and $p^2$ does not divide $d$. As $d$ and $p$
are relatively prime, $\mu(pd) = \mu(p)\mu(d)$ and the main term
becomes \bea S_{1;p}(X) & \ = \ &
\foh \sum_{d \le X/p \atop (2p,d)=1} \sum_{m^2|d} \mu(m) \nonumber\\
&=& \foh \sum_{m \le (X/p)^{1/2}\atop (2p,m)=1} \mu(m) \cdot
\sum_{d\ \le\ (X/p)/m^2 \atop (2p,d)=1} 1 \nonumber\\ &=& \foh
\sum_{m\le (X/p)^{1/2} \atop
(2p,m)=1} \mu(m)\left(\frac{X/p}{m^2} \cdot \frac{p-1}{2p} + O(1)\right) \nonumber\\
&=& \frac{(p-1)X}{4p^2} \sum_{m=1 \atop (2p, m)=1}^\infty
\frac{\mu(m)}{m^2} + O(X^{1/2}) \nonumber\\ &=& \fof
\frac{6}{\zeta(2)} \cdot \left(1 - \frac1{2^2}\right)^{-1} \cdot
\left(1-\frac1{p^2}\right)^{-1} \cdot \frac{(p-1)X}{p^2} + O(X^{1/2}) \nonumber\\
&=& \frac{2X}{(p+1)\pi^2} + O(X^{1/2}), \eea and the cardinality of
this piece is reduced by $(p+1)^{-1}$ (note above we used $\#\{n \le
Y: (2p,n)=1\}$ $=$ $\frac{p-1}{2p}Y+O(1)$). A similar analysis holds
for $S_{2;p}(X)$, as well as the even fundamental discriminants $d$
with $d/4 \equiv 2$ or $3$ modulo $4$).

We need to trivially modify the above arguments if $p\equiv 3 \bmod
4$. If for instance we require $d\equiv 1 \bmod 4$ then instead of
replacing $\mu(d)^2$ with $\mu(d)^2 (1 + \chi_4(d))/2$ we replace it
with $\mu(pd)^2 (1-\chi_4(d))/2$, and the rest of the proof proceeds
similarly.

It is a completely different story if $p=2$. Note if $d\equiv 1
\bmod 4$ then 2 \emph{never} divides $d$, while if $d/4 \equiv 2$ or
3 modulo 4 then 2 \emph{always} divides $d$. There are $3X/\pi^2 +
o(X^{1/2})$ even fundamental discriminants at most $X$, and $X/\pi^2
+ O(x^{1/2})$ of these are divisible by 2. Thus, if our family is
all even fundamental discriminants, we do get the factor of
$1/(p+1)$ for $p=2$, as one-third (which is $1/(2+1)$ of the
fundamental discriminants in this family are divisible by $2$.
\end{proof}

In our analysis of the terms from the $L$-functions Ratios
Conjecture, we shall need a partial summation consequence of Lemma
\ref{lem:numdwithpdivided}.

\begin{lem}\label{lem:partialsumdexpdpiX} Let $d$ denote an
even fundamental discriminant at most $X$ and $X^\ast = \sum_{d\le
X} 1$ and let $z = \tau - i w \frac{\log X}{2\pi}$ with $w \ge 1/2$.
Then \be \sum_{d\le X} e^{-2\pi i z\frac{\log(d/\pi)}{\log X}} \ = \
X^\ast e^{-2\pi i \left(1-\frac{\log \pi}{\log X}\right) z} \left(1
- \frac{2\pi i z}{\log X}\right)^{-1} + O(\log X). \ee
\end{lem}

\begin{proof} By Lemma
\ref{lem:numdwithpdivided} we have \be \sum_{d\le u} 1 \ = \
\frac{3u}{\pi^2} + O(u^{1/2}). \ee Therefore by partial summation we
have \bea & & \sum_{d\le X} e^{-2\pi i z \frac{\log(d/\pi)}{\log
X}}\nonumber\\ & \ = \ & e^{2\pi i \frac{\log
\pi}{\log X}} \sum_{d\le X} d^{-2\pi i z / \log X} \nonumber\\
&=& e^{2\pi i z\frac{\log \pi}{\log X}}
\left[\frac{3X+O(X^{1/2})}{\pi^2}\ X^{-\frac{2\pi i z}{\log X}} -
\int^X \left(\frac{3u}{\pi^2} + O(u^{1/2})\right) \cdot
u^{-\frac{2\pi
i z}{\log X}} \frac{-2\pi i z}{\log X} \frac{du}{u}\right]. \nonumber\\
\eea As we are assuming $w \ge 1/2$, the first error term is of size
$O(X^{1/2} X^{-w}) = O(1)$. The second error term (from the
integral) is $O(\log X)$ for such $w$. This is because the
integration begins at $1$ and the integrand is bounded by
$u^{-\foh-w}$. Thus \bea& & \sum_{d\le X} e^{-2\pi i z
\frac{\log(d/\pi)}{\log X}}\nonumber\\ & \ = \ &  e^{2\pi i
z\frac{\log \pi}{\log X}} \left[\frac{3X}{\pi^2} e^{-2\pi i z} +
\frac{\frac{3}{\pi^2}\cdot 2\pi i z}{\log X} \int^X u^{-2\pi
i z/\log X} du\right] +O(\log X) \nonumber\\
&=& e^{2\pi i z\frac{\log \pi}{\log X}} \left[\frac{3X}{\pi^2}
e^{-2\pi i z} + \frac{\frac{3}{\pi^2}\cdot 2\pi i z}{\log X}
\frac{X^{1-2\pi i z/\log X}}{1 - 2\pi i z/\log X}\right] + O(\log X)
\nonumber\\ &=& X^\ast e^{2\pi i z\frac{\log \pi}{\log X}} e^{-2\pi
i z} \left[1 + \frac{2\pi i z}{\log X} \sum_{\nu=0}^\infty
\left(\frac{2\pi i z}{\log
X}\right)^\nu\right] + O(\log X) \nonumber\\
&=& X^\ast e^{-2\pi i \left(1-\frac{\log \pi}{\log X}\right) z}
\left(1 - \frac{2\pi i z}{\log X}\right)^{-1} + O(\log X). \eea
\end{proof}


\section{Improved bound for non-square $m$ terms in
$S_M(X,Y,\hg,\Phi)$}\label{sec:improvingGaomnotsquare}

Gao \cite{Gao} proves that the non-square $m$-terms contribute $\ll
(U^2 Z \sqrt{Y} \log^7 X)/X$ to $S_M(X,Y,\hg,\Phi)$. As this bound
is just a little too large for our applications, we perform a more
careful analysis below. Denoting the sum of interest by $R$, \bea R
& \ = \ & \sum_{\alpha \le Z \atop (\alpha,2) = 1} \sum_{p \le Y
\atop (2\alpha,p) = 1} \frac{\log p}{p} \widehat{g}\left(\frac{\log
p}{\log X}\right) \sum_{m \neq 0, \Box} (-1)^m
\widetilde{\Phi}\left(\frac{mX}{2\alpha^2 p}\right)
\jsgeneral{m}{p}, \eea Gao shows that \be R \ \ll \ \sum_{\alpha \le
Z} \frac{\log^3 X}{\alpha^2} (R_1 + R_2 + R_3), \ee with \bea R_1,
R_2 & \ \ll \  & \frac{U \alpha^2 \sqrt{Y} \log^4 X}{X} \nonumber\\
\ \ R_3 \ \ & \ll & \frac{U^2 \alpha^2 \sqrt{Y} \log^7 X}{X}. \eea
The bounds for $R_1$ and $R_2$ suffice for our purpose, leading to
contributions bounded by $ (UZ\sqrt{Y}\log^4 X)/X$; however, the
$R_3$ bound gives too crude a bound -- we need to save a power of
$U$.

We have (see page 36 of \cite{Gao}, with $k=1$, $k_2=0$, $k_1 = 0$, $\alpha_1 = 1$
and $\alpha_0 = 0$) that \be R_3 \ \ll \ \int_1^Y \frac{X}{\alpha^2
V^{5/2}} \sum_{p < Y} \frac{\log p}{p^2} \sum_{m=1}^\infty (\log^3
m) m \widetilde{\Phi}'\left(\frac{mX}{2\alpha^2 p V}\right) dV. \ee

We have (see (3.10) of \cite{Gao}) that
\be\label{eq:boundsPhitildederiv} \widetilde{\Phi}'(\xi) \ \ll\
U^{j-1} |\xi|^{-j}\ {\rm for\ any\ integer\ j \ge 1}. \ee Letting $M
= X^{2008}$, we break the $m$-sum in $R_3$ into $m \le M$ and $m
> M$. For $m \le M$ we use \eqref{eq:boundsPhitildederiv} with $j=2$
while for $m > M$ we use \eqref{eq:boundsPhitildederiv} with $j=3$.
(Gao uses $j=3$ for all $m$. While we save a bit for small $m$ by
using $j=2$, we cannot use this for all $m$ as the resulting $m$ sum
does not converge.)

Thus the small $m$ contribute \bea & \ll \ & \int_1^Y
\frac{X}{\alpha^2 V^{5/2}} \sum_{p < Y} \frac{\log p}{p^2} \sum_{m
\le M} (\log^3 m) m \frac{U 2^2 \alpha^4 p^2 V^2}{m^2 X^2} dV
\nonumber\\ & \ll \ & \frac{U\alpha^2}{X} \sum_{p < Y} \log p
\sum_{m \le M} \frac{\log^3 m}{m} \int_1^Y \frac{dV}{\sqrt{V}}
\nonumber\\ & \ll \ & \frac{UY^{3/2} \alpha^2 \log^4 X}{X} \eea
(since $M = X^{2008}$ the $m$-sum is $O(\log^4 X)$). The large $m$
contribute \bea & \ll \ & \int_1^Y \frac{X}{\alpha^2 V^{5/2}} \sum_p
\frac{\log p}{p^2} \sum_{m > M} (\log^3 m) m \frac{U^2 2^3 \alpha^6
p^3 V^3}{m^3 X^3} dV \nonumber\\ &\ll \ & \frac{U^2\alpha^4}{X^2}
\sum_{p < Y} p \log
p \sum_{m > M} \frac{\log^3 m}{m^3} \int_1^Y V^{1/2} dV \nonumber\\
&\ll \ & \frac{\alpha^4 U^2 Y^{3/2} Y^2 \log X}{X^2 M^{2-\gep}}.
\eea For our choices of $U$, $Y$ and $Z$, the contribution from the
large $m$ will be negligible (due to the $M^{2-\gep} =
X^{4016-2\gep}$ in the denominator). Thus for these choices \bea R &
\ \ll \ & \sum_{\alpha \le Z} \frac{\log^3 X}{\alpha^2} (R_1 + R_2 +
R_3) \nonumber\\ & \ll & \frac{UZ\sqrt{Y}\log^7 X}{X} +
\frac{UZY^{3/2}\log^4 X}{X} + \frac{Z^3 U^2 Y^{7/2} \log^4
X}{X^{4018-2\gep}}. \eea The last term is far smaller than the first
two. In the first term we save a power of $U$ from Gao's bound, and
in the second we replace $U$ with $Y$. As $Y=X^\sigma$, for $\sigma$
sufficiently small there is a significant savings.


\ \\

\end{document}